\documentclass[a4paper, reqno]{amsart}
\usepackage[utf8]{inputenc}
\usepackage{graphicx}
\usepackage[bookmarksnumbered,colorlinks,plainpages]{hyperref}
\usepackage{rotating}
\usepackage{float}

\usepackage{amsthm}
\usepackage[all,2cell]{xy}
\usepackage{amsfonts}
\usepackage{amsmath}
\usepackage{amssymb}
\usepackage{xypic,leqno,amscd,amssymb,pstricks,latexsym,amsbsy,xypic,mathrsfs,verbatim}
\usepackage{multirow}
\usepackage{booktabs}
\usepackage{makecell}

\usepackage{etex}
\usepackage{mathtools}

\usepackage{tikz}
\usetikzlibrary{positioning,arrows}
\usetikzlibrary{decorations.pathreplacing}
\usepackage{tikz-cd}

\UseAllTwocells \SilentMatrices
\usepackage{amscd,amssymb,pstricks,latexsym,amsbsy,xypic,mathrsfs,verbatim}
\usepackage[all]{xy}
\usepackage{lscape}
\usepackage{enumerate}
\usepackage{graphicx}
\usepackage{bm}

\usepackage{amsmath}
\usepackage{xypic,leqno,amscd,amssymb,pstricks,latexsym,amsbsy,xypic,mathrsfs,verbatim}
\usepackage{multirow}
\usepackage{booktabs}
\usepackage{makecell}
\usepackage{mathrsfs}

\makeatletter
\newcommand*\bigcdot{\mathpalette\bigcdot@{.5}}
\newcommand*\bigcdot@[2]{\mathbin{\vcenter{\hbox{\scalebox{#2}{$\m@th#1\bullet$}}}}}
\makeatother

\makeatletter
\@namedef{subjclassname@2010}{
  \textup{2020} Mathematics Subject Classification}

\numberwithin{equation}{section}

\theoremstyle{plain}
\newtheorem{prop}{Proposition}[section]
\newtheorem{thm}[prop]{Theorem}
\newtheorem{cor}[prop]{Corollary}
\newtheorem{lem}[prop]{Lemma}
\newtheorem{defn}[prop]{Definition}

\theoremstyle{definition}
\newtheorem{exam}[prop]{Example}
\newtheorem{rem}[prop]{Remark}
\newtheorem{assump}[prop]{Assumption}

\def\End{{{\rm End}\,}}
\def\Hom{{{\rm Hom}\,}}

\def\add{{{\rm add}\,}}

\def\soc{{{\rm soc}\,}}
\def\top{{{\rm top}\,}}
\def\Ext{{{\rm Ext}\,}}

\def\D{{{\rm D}\,}}

\def\mod{{\text{\rm mod}}}
\def\Mod{{\text{\rm Mod}}}
\def\EG{{{\rm EG}}}

\def\Stab{{\rm Stab}}
\def\thick{{\rm thick}}
\def\tube{{\rm tube}}

\def\add{{\rm add}}

\def\Heartst3{{\rm Hearts \  _{\mathcal{S}_{\textit{i}}}\mathcal{H}_{\mathcal{U}_{\textit{i}}} \ of \ the \  bounded \ t-structure \ (\mathcal{D}^{\leq 0}, \mathcal{D}^{\geq 0})  \ on }}

\def\Dim{{ \rm dim }}

\def\Stab{{\rm Stab}}

\def\Hom{{\rm Hom}}
\def\Mod{{\rm mod}}
\def\End{{\rm End}}
\def\Ext{{\rm Ext}}

\def\Stab{{\rm Stab}}

\begin{document}
\title[Exchange graphs and Ext-quivers of hearts  of tube categories]{Exchange graphs and Ext-quivers of hearts  of tube categories}

\author[Mingfa Chen] {Mingfa Chen}

\address{Mingfa Chen;\; School of Mathematical Sciences, Xiamen University, Xiamen, 361005, Fujian, PR China.}
\email{mingfachen@stu.xmu.edu.cn}

\date{\today}
\makeatletter \@namedef{subjclassname@2020}{\textup{2020} Mathematics Subject Classification} \makeatother

\maketitle

\begin{abstract}
In this paper we introduce the notion of pre-simple-minded collection (pre-SMC) of type $\mathbb{A}$ in the bounded derived categories $\mathcal{D}^{b} (${\rm{\textbf{T}}}$_p)$ of tube categories $\textbf{T}_{p}$ of rank $p$. This provides an effective approach to  classify the hearts in $\mathcal{D}^{b} (${\rm{\textbf{T}}}$_p)$.
We then use this classification
to prove the exchange graph of hearts in $\mathcal{D}^{b} (${\rm{\textbf{T}}}$_p)$ is connnected.
Further, we classify the Ext-quivers of hearts in $\mathcal{D}^{b} (${\rm{\textbf{T}}}$_p)$.
As an application, we show that the space of  Bridgeland stability conditions on $\mathcal{D}^{b}(${\rm{\textbf{T}}}$_p)$ is connected.
\end{abstract}

\date{}
\maketitle

\medskip

\section{Introduction}

The notion of hearts of t-structures were introduced by Beilinson-Bernstein-Deligne \cite{bbd}, which become increasingly important in a wide range of research areas.
Bridgeland \cite{bris} proved a stability function with the Harder-Narashimhan  property on its heart
equivalently gives  a stability condition, and the space of stability conditions with its natural topology forms manifold.
For a triangulated category $\mathcal{D}$ and an exact endofunctor $\Phi$, Kim \cite{kim} introduced new entropy-type invariants using finite (co-)hearts of bounded (co-)t-structures to  describe the
categorical entropies of $\Phi\lvert_{\mathcal{A}}$ and $\Phi\lvert_{\mathcal{B}}$ for an ST-triple $(\mathcal{B},\mathcal{A},M)$.
Saor\'{i}n-$\rm{\check{S}}$t$\acute{}$ov\'{i}$\rm{\check{c}}$ek-Virili \cite{grot} proved any  compactly generated t-structure has Grothendieck hearts via strong and stable derivator for  well-generated algebraic or topological triangulated categories.
Research highlights about hearts are also studied in \cite{simp,geom,twis,
noet,sast,
tstr,wool}.

It turns out that both the study of individual hearts and the study of the whole collection of hearts provide us useful information on the category.
For a finite dimensional algebra $\Lambda$, K$\rm{\ddot{o}}$nig and Yang \cite{kysi} established
bijections among  finite length hearts, simple$-$minded collections in  $\mathcal{D}^{b}(\mod\Lambda)$, silting
objects in $\mathcal{K}^{b}(\rm{proj}\Lambda)$, and co$-$t$-$structures on $\mathcal{K}^{b}(\mathrm{proj}\Lambda)$.
Bachmann-Kong-Wang-Xu \cite{chow} described the Chow heart
as the category of even graded MU$_{2*}$MU-comodules for the Chow t-structure defined on the $\infty$-category of motivic spectra
over an arbitrary base field.
On the other hand, a great many research topics concern about  nilpotent representation categories of cyclic quivers (i.e.,  tube categories),  including but not limited to Hall algebras \cite{hall,para}, Cluster theories \cite{clusc}, GKM-theories \cite{gkm} and Quantum groups \cite{quan}.
It seems a lack of detailed description of hearts in the literature for the derived categories of tube categories.
Our motivation is to give a classification of hearts for such categories.

The exchange graph  of hearts in a triangulated category $\mathcal{D}$  encodes the
combinatorics of mutations of hearts in $\mathcal{D}$, which was
adapted by a great deal of mathematicians to different branches of mathematics \cite{brus,clus}.
For an acyclic quiver $Q$ and the Calabi-Yau-N
Ginzburg algebra $\Gamma_{N}Q$,
King and Qiu \cite{kingq} showed any heart in  the principal component of the oriented exchange graph is induced from some heart in  $\mathcal{D}^{b}(Q)$ via `Lagrangian immersion' $\mathcal{L}: \mathcal{D}^{b}(Q) \rightarrow \mathcal{D}^{b}(\Gamma_{N} Q)$.

The Ext-quivers  of  hearts  are of great importance in representation theories and cluster theories.
Assem and Happel \cite{ahge} gave a classification of repeatedly tilted algebras of $\mathbb{A}$-type in terms of quivers of gentle trees.
Qiu \cite{qext} generalized their result to classify the Ext-quivers of hearts of $\mathbb{A}$-type in terms of graded gentle trees,which provides the classification for Buan-Thomas' colored quivers for higher clusters of $\mathbb{A}$-type.

Exchange graphs and Ext-quivers  of hearts have played  key roles in the researches of the space of Bridgeland stability conditions \cite{bris}.
King and Qiu \cite{kingqi} introduced the cluster exchange groupoid associated to a nondegenerate
quiver with potential, and proved the space of stability conditions for the 3-Calabi-Yau category associated to the marked surface is simply connected.
Qiu \cite{qsta} proved the space of stability
conditions of any Dynkin quiver $Q$
is simply connected via the exchange graph and the Donaldson-Thomas type invariant associated to $Q$ can be calculated as a quantum dilogarithm function on its exchange graph.

In the present paper we study the exchange graphs and Ext-quivers of hearts in the bounded derived categories $\mathcal{D}^{b} (${\rm{\textbf{T}}}$_p)$ of tube categories $\textbf{T}_{p}$ of rank $p$.
We introduce the notion of pre-simple-minded collection of type $\mathbb{A}$ in $\mathcal{D}^{b} (${\rm{\textbf{T}}}$_p)$. We first provide a  classification of hearts in $\mathcal{D}^{b} (${\rm{\textbf{T}}}$_p)$ via this notion.
We investigate the mutations of simple-minded collections (SMCs) in $\mathcal{D}^{b} (${\rm{\textbf{T}}}$_p)$. We then use the classification of hearts to prove the exchange graph  for any tube category $\textbf{T}_{p}$ is connected.
Moreover we  introduce the notion of graded gentle one-cycle quiver.
We show  the Ext-quivers of  SMCs in $\mathcal{D}^{b} (${\rm{\textbf{T}}}$_p)$ are precisely the associated quivers of graded gentle one-cycle quivers with $p$ vertices, which is a generalization of graded cyclic quiver of simples in $\textbf{T}_p$.

The paper is organized as follows.
In Section 2 we collect the definitions  and properties of hearts and simple-minded collections in triangulated categories.
In  Section 3 we introduce the notion of pre-simple-minded collection of type $\mathbb{A}$ in $\mathcal{D}^{b} (${\rm{\textbf{T}}}$_p)$. We then classify all hearts in $\mathcal{D}^{b} (${\rm{\textbf{T}}}$_p)$ via this notation.
In Section 4 we investivate the connectedness property of  the exchange graph $\EG(\textbf{T}_p)$ for tube category $\textbf{T}_{p}$. Moreover we show the space of Bridgeland stability conditions on $\mathcal{D}^{b}(${\rm{\textbf{T}}}$_p)$ is connected via $\EG(\textbf{T}_p)$.
In  Section 5 we  introduce the notion of graded gentle one-cycle quiver, and describe all the Ext-quivers of  SMCs in $\mathcal{D}^{b} (${\rm{\textbf{T}}}$_p)$ via the graded gentle one-cycle quivers with $p$ vertices.

\emph{Notation.} Throughout this paper, let $\textbf{k}$ be an algebraically closed field, let $\mathcal{C}$ be an additive category.
All subcategories are assumed to be full and closed under isomorphisms.
If further $\mathcal{C}$ is abelian or triangulated and $\mathcal{S}$ is a set of objects in $\mathcal{C}$,  we write $\langle \mathcal{S} \rangle_{\mathcal{C}}$ for \emph{extension closure} of $\mathcal{S}$ which is the subcategory of $\mathcal{C}$ generated by the objects in $\mathcal{S}$ closed \emph{under extensions and direct summands},
and denote by $\add~ \mathcal{S}$ the subcategory of $\mathcal{C}$ consisting of direct summands of finite direct sums of objects in $\mathcal{S}$.
Assume  $\mathcal{C}$ is abelian, the \emph{right perpendicular category} $\mathcal{S}^{\bot_{\mathcal{C}}}$ of $\mathcal{S}$ in the sense of \cite{glri} is
$$\mathcal{S}^{\bot_{\mathcal{C}}}=\{X\in\mathcal{C}\mid\Hom(S, X)=0=\Ext^{1}(S, X)\ {\rm{for}}\ {\rm{all}}\ S\in\mathcal{S}\}.$$
Dually one defines \emph{left perpendicular category}.
Assume  $\mathcal{C}$ is triangulated and let $\thick(\mathcal{S})$ denote the smallest triangulated
subcategory of $\mathcal{C}$ containing objects in $\mathcal{S}$ and closed under taking direct summands. We say that
$\mathcal{S}$ is a set of generators of $\mathcal{C}$, or that $\mathcal{C}$ is generated by $\mathcal{S}$, when $\mathcal{C} = \thick(\mathcal{S})$. We put
$$\mathcal{S}^{\bot_{\mathcal{C}}}:=\{X\in\mathcal{C}\mid\Hom^{n}(S, X)=0\  {\rm{for}}\ {\rm{all}}\ S\in\mathcal{S}, n\in\mathbb{Z}\}.$$
Dually one defines  $^{\bot_{\mathcal{C}}}\mathcal{S}$.

\section{Preliminary}

In this section we recall homological properties for  tube categories and its bounded derived categories, and the definitions  and basic properties of hearts and simple-minded collections in triangulated categories.

\subsection{Tube categories}
The  \emph{tube}  category $\textbf{T}_p$ of rank $p\geq1$ is an abelian category in the sense of \cite[Sect.\;4.6]{ring} and \cite[Chap.\;X]{sims}, that is,
the nilpotent representation category of cyclic quiver $C_{p}$ with $p$ vertices, whose associated Auslander-Reiten (AR for short) quiver  $\Gamma(\textbf{T}_p)$ is a \emph{stable tube} of rank $p$.

Recall that  $\textbf{T}_p$ is a hereditary finite length abelian category with $p$ simple objects $S_{0}, \cdots, S_{p-1}$, equipped with an Auslander-Reiten (AR for short) translation $\tau$ satisfying $\tau(S_{i} )=S_{i-1}$, where the index is considered module $p$.
For any simple object $S_{j}$ in a tube category
$\textbf{T}_p$ and any $t\in \mathbb{Z}_{\geq 1}$, there is a unique object $S^{(t)}
_{j}$ of length $t$ and top $(S^{(t)}_{j}) = S_{j}$, and any indecomposable object in $\textbf{T}_p$ has this form. The tube category $\textbf{T}_p$ is a uniserial category in the sense
that all subobjects of $S^{(t)}_{j}$ form a chain with respect to the inclusion:
$$ 0:=S^{(0)}_{j-t}\subseteq S^{(1)}_{j-t+1}\subseteq S^{(2)}_{j-t+2}\subseteq \cdots \subseteq S^{(t-1)}_{j-1}\subseteq S^{(t)}_{j}.$$
Consequently, the socle of $S^{(t)}_{j}$ is given by soc $(S^{(t)}_{j})=S^{}_{j-t+1}$. Moreover, the \emph{composition factor sequence} of $S_{j}^{(t)}$ is given $(S_{j-t+1}, \cdots, S_{j-1}, S_{j})$, since $S^{(r)}_{j-t+r}/S^{(r-1)}_{j-t+r-1}=S^{}_{j-t+r}$ for $1\leq r\leq t$. The set of pairwise distinct simple objects appearing in this sequence is called the \emph{composition factor set} of $S_{j}^{(t)}$. It is easy to see that any $S_{j}^{(t)}$ with $t\geq p$ has the same composition factor set $\{S_0, S_2, \cdots, S_{p-1}\}$.

We denote by $u_{j,t}: S_{j-1}^{(t-1)} \rightarrow S_{j}^{(t)}$ and $p_{j,t}:  S_{j}^{(t)} \rightarrow S_{j}^{(t-1)}$ the irreducible injection map and surjection map respectively. For convenience, in the following we will simply denote them by $u$ and $p$
respectively if no confusions appear. For example, the composition
$$\xymatrix@C=3.9em@R=1ex@M=1.2pt@!0{
S_{j}^{(t)} \ar[rr]^-{p_{j,t}}& &S_{j}^{(t-1)} \ar[rr]^-{p_{j,t-1}}& &S_{j}^{(t-2)}\ar[rr]^-{p_{j,t-2}} &&\cdots\ar[rr]^-{p_{j,3}}&&S_{j}^{(2)}\ar[rr]^-{p_{j,2}}& & S_{j}^{}  \\ }
$$
will be just denoted by $p^{t-1}: S_{j}^{(t)}\rightarrow S_{j}^{}$. With this simplified notations, we have $u\circ p = p\circ u$
whenever it makes sense.

The AR-quiver $\Gamma(\textbf{T}_p)$ of the tube category $\textbf{T}_p$ has the shape of a `tube', which has an explicit description in \cite[Chap. X]{sims}.
We will provide a concrete example of $\Gamma(\textbf{T}_p)$ as in (\ref{ARquiveroftube}) in Section 3.

\subsection{Bounded derived categories of tube categories}
Following \cite{lenz} the bounded derived categories $\mathcal{D}^{b}(\textbf{T}_p)$ of tube categories $\textbf{T}_p$ is naturally equivalent  to the repetitive category
$\mathop{\bigvee}\limits_{k\in\mathbb{Z}}^{}\textbf{T}_p[k]$,
where each $\textbf{T}_p[k]$ is a copy of $\textbf{T}_p$, with objects written $S_{j}^{(t)}[k]$ for $S_{j}^{(t)}$ in $\textbf{T}_p$, and
morphisms given by
$$\Hom_{\mathcal{D}^{b}(\textbf{T}_p)}(S_{j_{1}}^{(t_{1})}[k_{1}], S_{j_{2}}^{(t_{2})}[k_{2}])\cong\Ext_{\textbf{T}_p}^{k_{2}-k_{1}}(S_{j_{1}}^{(t_{1})}, S_{j_{2}}^{(t_{2})}).$$
Here the expression $\mathop{\bigvee}\limits_{k\in\mathbb{Z}}^{}\textbf{T}_p[k]$ has two meanings. First it stands for the
additive closure $\add(\mathop{\cup}\limits_{k\in\mathbb{Z}}^{}\textbf{T}_p[k])$ of the union of all $\textbf{T}_p[k]$, and secondly it
indicates that there are no nonzero morphisms backwards, that is, from $\textbf{T}_p[k_{1}]$
to $\textbf{T}_p[k_{2}]$ for $k_{1}>k_{2}$.

We collect fundamental exact sequences and non-zero hommorphisms in the tube category $\textbf{T}_p$, see for examples \cite[Thm.\;2.2]{sims},  \cite[Lem.\;A.1]{rw} and \cite[Prop.  4.3]{stab}.
\begin{lem}\label{exact seqs in tube}
The following are exact sequences in {\rm{\textbf{T}}$_p$} for any $j \in \mathbb{Z}/p\mathbb{Z}$ and $t\geq1$:
$$\xymatrix@C=4.9em@R=5ex@M=2.5pt@!0{
\quad (1)\ \ 0 \ar[r]& S_{j-1}^{(t)} \ar[rr]^{u}& &S_{j}^{(t+1)} \ar[rr]^-{p^{t}}& & S_{j}^{}\ar[r]& 0; & & & & & & & & &  & & & & & & \\
\quad (2)\ \ 0 \ar[r]& S_{j-t}^{} \ar[rr]^-{u^{t}}& &S_{j}^{(t+1)} \ar[rr]^-{p}& & S_{j}^{(t)}\ar[r]& 0; & & & & & & & & &  & & & & & & \\
\quad (3)\ \ 0 \ar[r]& S_{j}^{(t)} \ar[rr]^-{(u,p)^{t}}&  &S_{j+1}^{(t+1)}\oplus S_{j}^{(t-1)} \ar[rr]^-{(p,-u)}&  & S_{j+1}^{(t)}\ar[r]& 0; & & & & & & & & &  & &  \\
\quad (4)\ \ 0 \ar[r]^-{} &  S_{j-t+1}^{} \ar[rr]^-{u^{t-1}}& &S_{j}^{(t)} \ar[rr]^-{u\circ p = p\circ u}& & S_{j+1}^{(t)}\ar[r]^-{p^{t-1}}  & S_{j+1}^{}\ar[r] & 0.& & & \\
  }
$$
\end{lem}

\begin{lem}\label{non-zero morphism between objects}
For any objects $S_{j_i}^{(t_i)}\; (i=1,2)$ in {\rm{\textbf{T}}$_p$}, $\Hom(S_{j_1}^{(t_1)}, S_{j_2}^{(t_2)})\neq 0$ if and only if the following hold:
\begin{itemize}
  \item [(1)] $\top(S_{j_1}^{(t_1)})$ belongs to the composition factors set of $S_{j_2}^{(t_2)}$;
  \item [(2)] $\soc(S_{j_2}^{(t_2)})$ belongs to the composition factors set of $S_{j_1}^{(t_1)}$.
\end{itemize}
\end{lem}

\subsection{Hearts in triangulated categories}
Let $\mathcal{D}$ be a triangulated category. Following \cite{bbd} a \emph{t-structure}  on $\mathcal{D}$ is a pair $(\mathcal{D}^{\leq 0}, \mathcal{D}^{\geq 0})$ of strictly full subcategories $(\mathcal{D}^{\leq n}:= \mathcal{D}^{\leq 0}[-n], \mathcal{D}^{\geq n}:= \mathcal{D}^{\geq 0}[-n])$ such that
\begin{itemize}
\item [(1)] $\Hom(\mathcal{D}^{\leq 0}, \mathcal{D}^{\geq 1} )= 0;$

\item [(2)] $\mathcal{D}^{\leq -1} \subset \mathcal{D}^{\leq 0}, \mathcal{D}^{\geq 1} \subset \mathcal{D}^{\geq 0};$

\item [(3)] For any object $X$ in $\mathcal{D}$, there exist a triangle  $A \rightarrow X \rightarrow B \rightarrow A[1]$ with $A \in \mathcal{D}^{\leq 0} , B \in \mathcal{D}^{\geq 1}.$
\end{itemize}
The \emph{heart} $\mathcal{A} :=\mathcal{D}^{\leq 0} \cap \mathcal{D}^{\geq 0}$ of $(\mathcal{D}^{\leq 0}, \mathcal{D}^{\geq 0})$ is always abelian. The t-structure $(\mathcal{D}^{\leq 0}, \mathcal{D}^{\geq 0})$ is said to be \emph{bounded} if
$$ \mathop{\cup}  \limits_{ n \in \mathbb{Z} }^{} \mathcal{D}^{\leq 0}[n]  = \mathcal{D} = \mathop{\cup}  \limits_{ n \in \mathbb{Z} }^{} \mathcal{D}^{\geq 0}[n].$$

In this paper, we only consider bounded t-structures and their hearts.

\subsection{Simple-minded collections}
Following  \cite{kysi,simp} a collection $\mathcal{X}=  \{  X_{1},\ \dots,\ X_{r}  \}$ of objects in $\mathcal{D}$ is said to be a \emph{simple-minded collection} (\emph{SMC} for short) if the following conditions hold for $i, j = 1,\ \dots,\ r:$
\begin{itemize}
\item [(1)] $\Hom(X_{i}, X_{j} [k]) = 0,\,\forall\,k < 0$;

\item [(2)] $\End(X_{i})$ is a division algebra and $\Hom(X_{i}, X_{j})$ vanishes for $i\not= j$;

\item [(3)] $\mathcal{D}=$ thick $(X_{1},\ \dots,\ X_{r})$, i.e., $X_{1},\ \dots,\ X_{r}$ generate $\mathcal{D}$.
\end{itemize}

The following result connects the relationship between hearts and SMCs, cf. \cite[\rm{Thm.\;7.12}]{kysi}.
\begin{lem}\label{bijectionHeartandSmc}
For a Hom-finite Krull-Schmidt triangulated category $\mathcal{D}$, there is a bijection between the set of the finite length hearts of bounded t-structures and the set of SMCs which commute with mutations.
\end{lem}

\begin{rem}
In general,  it  seems a lack of
detailed description of SMCs in  a  triangulated category.
In next section, we will provide a classification of SMCs in $\mathcal{D}^{b} (\textbf{T}_p)$ via smaller SMCs  in the triangulated subcateroties of type $\mathbb{A}$.
\end{rem}

\section{Pre-smcs of type $\mathbb{A}$ and  hearts in $\mathcal{D}^{b} (${\rm{\textbf{T}}}$_p)$ }
In this  section we introduce the notion of pre-simple-minded collection (pre-SMC) of type $\mathbb{A}$  in $\mathcal{D}=\mathcal{D}^{b} (\textbf{T}_p)$.
We obtain a classification of the hearts in $\mathcal{D}^{b} (\textbf{T}_p)$ via pre-SMCs of type $\mathbb{A}$.

Given integers $1\leq l_{i}<p$ and $a_{i}$ such that $S_{a_{i}}\in\textbf{T}_p$ which  $a_{i}$ is considered module $p$.  We write $\mathcal{A}_{l_{i}}=\langle  \tau^{j}S_{a_{i}} \ | \ 0 \leq j < l_{i}  \rangle$ for a subcategory of $\textbf{T}_p$. It follows that $\mathcal{A}_{l_{i}}\simeq \Mod\;  \mathbb{A}_{l_{i}}$ as an equivalence of categories, where $\mathbb{A}_{l_{i}}$ means Dynkin quiver algebra of type $\mathbb{A}$ with $l_{i}$ simples.
Let $\mathcal{S}$ be a (possibly empty) proper collection of simple objects in $\mathcal{A}=\textbf{T}_p$.
For simplification, we represent $\mathcal{S}$ with non-empty by
\begin{equation}\label{non-empty of simples}
\mathcal{S}=\mathop{\cup}\limits_{i=1 }^{n}\{\tau^{k}S_{a_{i}} \mid0\leq k < l_{i}\}
\end{equation}
with $a_{0}=a_{n}-p$ and $a_{i-1}<a_{i}$.
There is an equivalence
$\langle \mathcal{S}\rangle_{\mathcal{A}}  = \mathop{\coprod} \limits_{ i=1}^{n}\, \mathcal{A}_{l_{i}}  \simeq  \mathop{\coprod} \limits_{i=1 }^{n}  \Mod\,   \mathbb{A}_{l_{i}}$ and $\langle \mathcal{S}\rangle_{\mathcal{D}}  =  \mathop{\coprod} \limits_{i=1 }^{n}\, \mathcal{D}^{b}(\mathcal{A}_{l_{i}})$. The set $\mathcal{S}$ in $\Gamma(\textbf{T}_p)$ has the following form, where the objects in $\mathcal{S}$ have been marked by \xymatrix@C=3.1em@R=8.5ex@M=3.8pt@!0{ *+[F]{\;\;} \\ }\,.
\begin{equation} \label{ARquiveroftube}
\xymatrix@C=1.8em@R=2.8ex@M=0.3pt@!0{
&&\ar@{--}[ddddd]&&&&&&&&&&&&&&&&&&&&\ar@{--}[ddddd]&&&&&&&\\
&&&&&&\vdots&&&&&&\vdots&&&&&&&\vdots&&&&&&\\
&&&&&&&&&&&&\mathcal{A}_{l_{i}}&&&&&&&&&&&&&&&&\\
&&&\cdots&&&\mathcal{A}_{l_{i-1}}&&\quad\cdots&&&&\ar@{->}[rd]&&&\cdots&&&&
\mathcal{A}_{l_{i+1}}&&\quad\cdots&&&&\\
&&&&&&\ar@{->}[rd]&&&&&
\ar@{->}[ru]\ar@{->}[rd]&&\ar@{->}[rd]&&
\ar@{->}[rd]&&
\ar@{->}[rd]&&\ar@{->}[rd]&&&&&&&\\
&&S_{0}&\cdots&&*+[F]{\tau S_{a_{i-1}}}\ar@{->}[ru]&&*+[F]{S_{a_{i-1}}}&\quad\cdots&&
*+[F]{\tau^{2}S_{a_{i}}}\ar@{->}[ru]&&*+[F]{\tau S_{a_{i}}}\ar@{->}[ru]&&*+[F]{S_{a_{i}}}\ar@{->}[ru]&&
\tau^{-1}S_{a_{i}}\ar@{->}[ru]&&
*+[F]{\tau S_{a_{i+1}}}\ar@{->}[ru]&&*+[F]{S_{a_{i+1}}}&\quad\cdots&
S_{0}&&&&&&&\\  }
\end{equation}

Given $\mathcal{S}= \mathop{\cup}\limits_{ i=1 }^{n}\{  \tau^{k}S_{a_{i}}\mid0\leq k < l_{i}\}$ as in (\ref{non-empty of simples}), we denote $F$ by the following assignment
$$F\,:\langle \mathcal{S} \rangle\longrightarrow \langle \mathcal{S}\cap\tau\mathcal{S} \rangle\cup \left( \langle\tau\mathcal{S}[1]\rangle\setminus\langle \mathcal{S}[1]\cap\tau\mathcal{S}[1]\rangle\right)$$
satisfying
\begin{equation}\label{F action in type A}
F(S^{(t)}_{j})=
\begin{cases}
S^{(l_{i}+1-t_{})}_{a_{i}-t}[1], & \text{if $\top (S^{(t)}_{j})=S^{}_{a_{i}};$}\\
S^{(t)}_{j}, & \text{if else.}\\
\end{cases}
\end{equation}

Let $l(0)=0$ and $l(i)=l(i-1)+l_{i}$\,$(1\leq i\leq n)$.
Further we represent a SMC\, $\mathcal{U}_{i}$ in $\mathcal{D}^{b}(\mathcal{A}_{l_{i}})$ as
$\mathcal{U}_{i}=\{ S^{(t_{m})}_{j_{m}}[k^{}_{m}] \ |\ l(i-1)+1\leq m\leq l(i-1)+l_{i}\}$ with $S^{(t_{m})}_{j_{m}}\in\mathcal{A}_{l_{i}}$ and $k^{}_{m}\in\mathbb{Z}$.
Let $l=l(n)$. Then the collection $\mathcal{U}:=\mathop{\cup} \limits_{ i=1 }^{n}\mathcal{U}_{i}=\{ S^{(t_{m})}_{j_{m}}[k_{m}] \ |\ 1\leq m\leq l \}$ is a SMC in $\langle \mathcal{S} \rangle_{\mathcal{D}}$ and $S^{(l_{i}+1)}_{a_{i}}, S_{b}\in \mathcal{U}^{\bot_{\mathcal{D}}}\,(1\leq i\leq n, a_{i-1}<b<a_{i}-l_{i})$.
Let $F(\mathcal{U}_{i})=\mathop{\cup}\limits_{ m=l(i-1)+1 }^{l(i-1)+l(i)} \{F(S^{(t_{m})}_{j_{m}})[k^{}_{m}]\mid k^{}_{m}\geq0 \}\dot\cup\{S^{(t_{m})}_{j_{m}}[k^{}_{m}]\mid k^{}_{m}<0\}$,
and $\mathcal{X}_{i}$ be the  following (disjoint) collection of objects
\begin{equation}\label{F(U_i)and X_i}
\mathcal{X}_{i}=\{S^{(l_{i}+1)}_{a_{i}},  S_{b}\mid a_{i-1}<b<a_{i}-l_{i} \}\cup F(\mathcal{U}_{i}).
\end{equation}
Denote by
\begin{equation}\label{collection of X}
\mathcal{X}:=
\begin{cases}
\{ S_{0}, S_{1}, \cdots, S_{p-1}  \}, & \text{if $\mathcal{S}=\emptyset;$}\\ \mathop{\cup}\limits_{ i=1 }^{n}\mathcal{X}_{i}, & \text{if else.}\\
\end{cases}
\end{equation}

Clearly, $\mathcal{X}=\{ S_{0}, S_{1}, \cdots, S_{p-1}\}$ is a SMC in $\mathcal{D}^{b} (${\rm{\textbf{T}}}$_p)$.
From now onward, we always consider non-empty  $\mathcal{S}= \mathop{\cup} \limits_{ i=1 }^{n}\{ \tau^{k}S_{a_{i}}\mid0\leq k < l_{i}\}$ as in (\ref{non-empty of simples}) and keep notations as above in the rest of paper.

In the following we show  each collection
$\mathcal{X}=\mathop{\cup}\limits_{ i=1 }^{n}\mathcal{X}_{i}$ as in (\ref{collection of X}) is a SMC in $\mathcal{D}^{b} (${\rm{\textbf{T}}}$_p)$.

\begin{lem} \label{vanising homs}
Keep notations as in (\ref{collection of X}).
Let $S^{(t_{})}_{j}\in\langle \mathcal{S}\rangle_{\mathcal{D}}$, $S^{(t_{m})}_{j_{m}}[k^{}_{m}]\in\mathcal{U}$ and $S^{(l_{i}+1)}_{a_{i}}, S_{b}\in \mathcal{U}^{\bot_{\mathcal{D}}}\,(1\leq i\leq n, a_{i-1}<b<a_{i}-l_{i})$.
Then
\begin{itemize}
\item [(1)] $\Hom^{\bullet}(S^{(t_{})}_{j}, S^{(l_{i}+1)}_{a_{i}})=0=\Hom^{\bullet}(S^{(t_{})}_{j}, S^{}_{b})$ with $j\neq a_{i}$;
\item [(2)] $\Hom^{\bullet}(S^{(l_{i}+1)}_{a_{i}}, S^{(t_{})}_{j})=0=\Hom^{\bullet}(S^{}_{b}, S^{(t_{})}_{j})$ with $j\neq a_{i}$;
\item [(3)] $\Hom^{\bullet}(S^{(l_{i}+1)}_{a_{i}}, F(S^{(t_{})}_{a_{i}}))=0
=\Hom^{\bullet}(S^{}_{b}, F(S^{(t_{})}_{a_{i}}))$ with $j= a_{i}$;
\item [(4)] $\Hom^{d}(S^{(t_{m'})}_{j_{m'}}[k^{}_{m'}], F(S^{(t_{m})}_{j_{m}})[k^{}_{m}])=0$ for any $d\leq0$, $k^{}_{m'}<0$ and $k^{}_{m}\geq0$.
\end{itemize}
\end{lem}

\begin{proof}
We first prove (1). For any objects $S^{(t_{1})}_{j_{1}},   S^{(t_{2})}_{j_{2}}$ in $\textbf{T}_{p}$, since $\Hom^{\bullet}(S^{(t_{1})}_{j_{1}},  S^{(t_{2})}_{j_{2}})=0$ if and only if $\Hom^{0}(S^{(t_{1})}_{j_{1}},  S^{(t_{2})}_{j_{2}})=0$$=\Hom^{1}(S^{(t_{1})}_{j_{1}},  S^{(t_{2})}_{j_{2}})$, it suffices to prove the latter. Since $S^{(t_{})}_{j}\in\langle \mathcal{S}\rangle_{\mathcal{D}}$ with $j\neq a_{i}$, we have $a_{i}-l_{i}<j<a_{i}$ and $a_{i}-l_{i}<j-t+1<a_{i}$ for some $1\leq i\leq n$. Note that $\soc (S^{(l_{i}+1)}_{a_{i}})=S_{a_{i}-l_{i}}$,
$\top (S^{(l_{i}+1)}_{a_{i}})=S_{a_{i}}$. Then
$\Hom^{0}(S^{(t_{})}_{j}, S^{(l_{i}+1)}_{a_{i}})=0=\Hom^{1}(S^{(t_{})}_{j}, S^{(l_{i}+1)}_{a_{i}})=\D\Hom^{0}(S^{(l_{i}+1)}_{a_{i}}, S^{(t_{})}_{j-1})$ by Lemma \ref{non-zero morphism between objects} since $\soc (S^{(l_{i}+1)}_{a_{i}})$ (resp. $\top (S^{(l_{i}+1)}_{a_{i}})$) does not belong to the composition factor set of $S^{(t_{})}_{j}$ (resp. $S^{(t_{})}_{j-1}$).
Similarly, one can prove (2) and (3).

Assume $S^{(t_{m'})}_{j_{m'}}[k^{}_{m'}], S^{(t_{m})}_{j_{m}}[k^{}_{m}]\in\mathcal{U}$ with $k^{}_{m'}<0$ and $k^{}_{m}\geq0$.
Then $\Hom^{d}(S^{(t_{m'})}_{j_{m'}}[k^{}_{m'}], S^{(t_{m})}_{j_{m}}[k^{}_{m}])=0$.
If $\top (S^{(t_{m})}_{j_{m}})\neq S^{}_{a_{i}}$ for any $1\leq i\leq n$, then $\Hom^{d}(S^{(t_{m'})}_{j_{m'}}[k^{}_{m'}], F(S^{(t_{m})}_{j_{m}})[k^{}_{m}])=0$ since $F(S^{(t_{m})}_{j_{m}}[k^{}_{m}])=S^{(t_{m})}_{j_{m}}[k^{}_{m}]$.
Otherwise, $\top (S^{(t_{m})}_{j_{m}})= S^{}_{a_{i}}$ for some $1\leq i\leq n$ and $F(S^{(t_{m})}_{a_{i}})=S^{(l_{i}+1-t_{m})}_{a_{i}-t_{m}}[1]$.
Applying $\Hom^{d}(S^{(t_{m'})}_{j_{m'}}[k^{}_{m'}], -)$ to the following triangle
$$S^{(l_{i}+1)}_{a_{i}} \rightarrow S^{(t_{m})}_{a_{i}} \rightarrow S^{(l_{i}+1-t_{m})}_{a_{i}-t_{m}}[1]\rightarrow S^{(l_{i}+1)}_{a_{i}}[1],$$
we get $\Hom^{d}(S^{(t_{m'})}_{j_{m'}}[k^{}_{m'}], F(S^{(t_{m})}_{j_{m}})[k^{}_{m}])=0$ since $\Hom^{\bullet}(S^{(t_{m'})}_{j_{m'}}[k^{}_{m'}], S^{(l_{i}+1)}_{a_{i}})=0$ by similar arguments as in (1).
This proves (4).
\end{proof}

By similar arguments as in Lemma \ref{vanising homs}, we have the following result.

\begin{lem} \label{isom homs}
Keep notations as in (\ref{collection of X}).
Let $S^{(t_{1})}_{j_{1}},  S^{(t_{2})}_{j_{2}}\in\langle \mathcal{S}\rangle_{\mathcal{D}}$.
Then
\begin{equation}\label{isom homseq}
\Hom^{\bullet}(F(S^{(t_{1})}_{j_{1}}),\, F(S^{(t_{2})}_{j_{2}}))\cong
\Hom^{\bullet}(S^{(t_{1})}_{j_{1}},\, S^{(t_{2})}_{j_{2}}).
\end{equation}
Further, if $\mathcal{S}=\{S_{1}, \cdots,  S_{p-1}\}$ and $t_{1}\geq t_{2}$, then
\begin{equation}\label{isom homseq of p-1 S}
\Hom^{\bullet}(F(S^{(t_{1})}_{p-1}),\, S^{(t_{2})}_{p-1})\cong
\Hom^{\bullet}(S^{(t_{1})}_{p-1},\, S^{(t_{2})}_{p-1}).
\end{equation}
\end{lem}

\begin{lem} \label{thick condition}
$\mathcal{D}^{b} (${\rm{\textbf{T}}}$_p)=\thick(\mathcal{X})$ with $\mathcal{X}=\mathop{\cup}\limits_{ i=1 }^{n}\mathcal{X}_{i}$ as in (\ref{collection of X}).
\end{lem}

\begin{proof}
By the triangles
\begin{equation}\label{exchange triangle}
S^{(l_{i}+1)}_{a_{i}} \rightarrow S^{(t)}_{a_{i}} \rightarrow S^{(l_{i}+1-t_{})}_{a_{i}-t}[1]\rightarrow S^{(l_{i}+1)}_{a_{i}}[1],\end{equation}
we have $S^{(t)}_{a_{i}}\in\thick(S^{(l_{i}+1)}_{a_{i}}, F(S^{(t)}_{a_{i}}))$ for any $1\leq i\leq n$ and $1\leq t\leq l_{i}$.
It follows that $\langle \mathcal{S} \rangle_{\mathcal{D}}\subset\thick(\mathcal{X})$ since
$\mathcal{U}\subset\thick(\mathcal{X})$ is a SMC in $\langle \mathcal{S} \rangle_{\mathcal{D}}$ with  $\thick(\mathcal{U})=\langle \mathcal{S}\rangle_{\mathcal{D}}$.
By the triangles
$$S^{(l_{i})}_{a_{i}}[-1]\rightarrow S^{}_{a_{i}-l_{i}}\rightarrow S^{(l_{i}+1)}_{a_{i}}\rightarrow S^{(l_{i})}_{a_{i}},$$
we have $S_{j}\in\thick(\mathcal{X})$ with $a_{i}-l_{i}\leq j\leq a_{i}$ since $S^{(l_{i})}_{a_{i}}\in\langle\mathcal{S}\rangle_{\mathcal{D}}$.
Note that $S_{b}\in\thick(\mathcal{X})$ with $a_{i-1}<b<a_{i}-l_{i}$.
Hence $S_{j}\in\thick(\mathcal{X})$ with $0\leq j\leq p-1$
and {\rm{\textbf{T}}}$_p\in\thick(\mathcal{X})$.
It follows that $\thick(\mathcal{X})=$ $\mathcal{D}^{b} (${\rm{\textbf{T}}}$_p)$.
\end{proof}

\begin{lem} \label{construct SMCs}
The collection
$\mathcal{X}=\mathop{\cup}\limits_{ i=1 }^{n}\mathcal{X}_{i}$ as in (\ref{collection of X}) is a SMC in $\mathcal{D}^{b} (${\rm{\textbf{T}}}$_p)$.
\end{lem}

\begin{proof}
Recall that $\mathcal{U}$ is a SMC in $\langle \mathcal{S} \rangle_{\mathcal{D}}$ and $\Hom^{d}(S^{(t_{1})}_{j_{1}}, S^{(t_{2})}_{j_{2}}[-1])=0$ for any $S^{(t_{1})}_{j_{1}}, S^{(t_{2})}_{j_{2}}\in{\rm{\textbf{T}}}_p$ and $d\leq 0$.
Combining the above arguments with Lemmas \ref{vanising homs} and \ref{isom homs}, we have $\Hom^{d}(X_{i}, X_{j})=0$ for any
$X_{i}, X_{j}\in\mathcal{X}$ and $d\leq 0$.
Note that $\End(S^{(t)}_{j})\cong\textbf{k}$ for any $t\leq p$ and  $0\leq j\leq p-1$.  Cleary, we have $t_{m}<p$ with  $S^{(t_{m})}_{j_{m}}[k^{}_{m}]\in\mathcal{U}$ and $l_{i}+1\leq p$ with  $S^{(l_{i}+1)}_{a_{i}}\in\mathcal{X}$.
It follows that $\End(X_{j}, X_{j})\cong\textbf{k}$ for any
$X_{j}\in\mathcal{X}$.
Moreover we have $\thick(\mathcal{X})=$ $\mathcal{D}^{b} (${\rm{\textbf{T}}}$_p)$ by Lemma \ref{thick condition}.
Hence, $\mathcal{X}$ is a SMC in $\mathcal{D}^{b} (${\rm{\textbf{T}}}$_p)$.
\end{proof}

The following  lemma describes all bounded t-structures on $\mathcal{D}^{b} ({\rm{\textbf{T}}}_{p})$.

\begin{lem} \label{recollements for hearts of tubes}
$($\cite[\rm{Prop.\;2.30}]{boun}$)$
Given a bounded t-structure $(\mathcal{D}^{\leq 0}, \mathcal{D}^{\geq 0})$ on $\mathcal{D}=\mathcal{D}^{b} (\emph{\textbf{T}}_{p})$, there is unique (up to equivalence) proper collection $\mathcal{S}$ of simple objects in $\emph{\textbf{T}}_{p}$ such that
\begin{itemize}
\item [(1)] $(\mathcal{D}^{\leq 0}, \mathcal{D}^{\geq 0})$ is compatible with the recollement $\mathcal{R}_{\mathcal{S}}$
$$
\xymatrix@C=3.3em@R=6.7ex@M=1pt@!0{
\langle\mathcal{S}\rangle^{\bot_{\mathcal{D}}}\ar@{->}[rr]|-{i_*}&&
\mathcal{D}^{b}(\emph{\textbf{T}}_{p})\ar@/_2.5ex/[ll]|-{i^{*}}
\ar@/^2.5ex/[ll]|-{i^{!}}\ar@{->}[rr]|-{j^{*}}&&
\langle\mathcal{S}\rangle_{\mathcal{D}}
\ar@/_2.5ex/[ll]|-{j_{!}}\ar@/^2.5ex/[ll]|-{j_{*}}&&&\\ }
$$
where $i_*,\ j_{!}$ are the inclusion functors.
\item [(2)] the corresponding t-structure on $\langle\mathcal{S}\rangle^{\bot_{ \mathcal{D} }}$ has heart $\mathcal{S}^{\bot_{\emph{\textbf{T}}_{p} }}[m]$ for some $m$.
\end{itemize}
In particular, each bounded t-structure on $\mathcal{D}^{b} (\emph{\textbf{T}}_{p})$ has finite length heart.
\end{lem}

\begin{lem}  ~\label{calculation lemma}
Keep notations as in (\ref{collection of X}) and Lemma \ref{recollements for hearts of tubes}. Let $S^{(t)}_{j} \in\langle\mathcal{S}\rangle_{\mathcal{D}}$.  Then
\begin{itemize}
\item [(1)]  $j_{*}S^{(t)}_{j}=S^{(t)}_{j}=F(S^{(t)}_{j})$ with $j\not=a_{i}$ for any $i$;
\item [(2)]  $j_{*}S^{(t)}_{a_{i}}=S^{(l_{i}+1-t)}_{a_{i}-t} [1]=F(S^{(t)}_{a_{i}})$ with $j=a_{i}$ for some $i$;
\end{itemize}
\end{lem}

\begin{proof}
Let $Y\in\langle\mathcal{S}\rangle_{\mathcal{D}}\subset\mathcal{D}^{b} (\mathcal{T}_{p})$. Then there is a (unique) triangle
$i_{*}i^{!}Y \stackrel{}{\rightarrow}Y \stackrel{}{\rightarrow} j_{*}j^{*}Y \rightarrow i_{*}i^{!}Y[1]$
such that  $i_{*}i^{!}Y\in\langle\mathcal{S}\rangle^{\bot_{\mathcal{D} }}$  and $j_{*}Y\cong j_{*}j^{*}Y\in(\langle\mathcal{S}\rangle^{\bot_{ \mathcal{D} }})^{\bot_{\mathcal{D}}}$.
Note that $S^{(l_{i}+1-t)}_{a_{i}-t}[1], S^{(t)}_{j}\in(\langle \mathcal{S}\rangle^{\bot_{\mathcal{D}}})^{\bot_{\mathcal{D}}}$ with $j \not=a_{i}$.
Therefore, by the triangles
$$S^{(l_{i}+1)}_{a_{i}} \rightarrow S^{(t)}_{a_{i}}\rightarrow S^{(l_{i}+1-t)}_{a_{i}-t} [1] \rightarrow S^{(l_{i}+1)}_{a_{i}}[1]\quad\quad\textit{and}\quad\quad
0\rightarrow S^{(t)}_{j} \rightarrow S^{(t)}_{j} \rightarrow0,$$
we get $j_{*}S^{(t)}_{a_{i}}=S^{(l_{i}+1-t)}_{a_{i}-t}[1]$ and
$j_{*}S^{(t)}_{j}=S^{(t)}_{j}$.
We are done.
\end{proof}

Note that the shift functor $[1]$ is an auto-equivalence of $\mathcal{D}^{b} ({\textbf{T}_p})$.
Hence, for any $k\in\mathbb{Z}$, $\mathcal{H}$ is a heart in $\mathcal{D}^{b} ({\textbf{T}_p})$ if and only if $\mathcal{H}[k]$ is a heart.
Therefore, in the following we only classify hearts up to shift actions.

We first describe all SMCs in $\mathcal{D}^{b}(\textbf{T}_p)$.

\begin{prop} \label{heart forms}
Up to shift actions of $\mathcal{D}^{b} (${\rm{\textbf{T}}}$_p)$, the collection $\mathcal{X}$ of objects is a SMC in $\mathcal{D}^{b}(${\rm{\textbf{T}}}$_p)$ if and only if  $\mathcal{X}=\{ S_{0}, \cdots, S_{p-1}\}$ or $\mathcal{X}=\mathop{\cup}\limits_{ i=1 }^{n}\mathcal{X}_{i}$ as in (\ref{collection of X}).
\end{prop}

\begin{proof}
Cleary, each of $\mathcal{X}=\{S_{0}, S_{1}, \cdots, S_{p-1}\}$ and $\mathcal{X}=\mathop{\cup}\limits_{ i=1 }^{n}\mathcal{X}_{i}$ is a SMC in $\mathcal{D}^{b}(${\rm{\textbf{T}}}$_p)$ by Lemma \ref{construct SMCs}.
On the other hand, up to shift actions of $\mathcal{D}^{b} (${\rm{\textbf{T}}}$_p)$, each heart $\mathcal{H}$ of t-structure $(\mathcal{D}^{\leq 0}, \mathcal{D}^{\geq 0})$ on $\mathcal{D}^{b}(\textbf{T}_p)$ is compatible with the recollement $\mathcal{R}_{\mathcal{S}}$ w.r.t a proper collection $\mathcal{S}$ of simple objects in $\textbf{T}_{p}$ by Lemma \ref{recollements for hearts of tubes} such that
$$\mathcal{D}^{\leq0}=\{X\in\mathcal{D}^{b}(\textbf{T}_p)\mid i^{*}X\in\mathcal{T}^{\leq0}, j^{*}X\in\mathcal{Y}^{\leq0}\}\quad\textit{and}\quad
\mathcal{D}^{\geq0}=\{X\in\mathcal{D}^{b}(\textbf{T}_p)\mid i^{!}X\in\mathcal{T}^{\geq0}, j^{*}X\in\mathcal{Y}^{\geq0}\},$$
with $(\mathcal{T}^{\leq0}, \mathcal{T}^{\geq0})$ is a t-structure on $\langle \mathcal{S} \rangle ^{\bot_{ \mathcal{D}}}$ and $(\mathcal{Y}^{\leq0}, \mathcal{Y}^{\geq0})$ is a t-structure
on $\langle\mathcal{S}\rangle_{\mathcal{D}}$ with heart $\mathcal{Y}$.
If $\mathcal{S}=\emptyset$, then $(\mathcal{D}^{\leq 0}, \mathcal{D}^{\geq 0})$ is a standard t-structure on $\mathcal{D}^{b}(\textbf{T}_p)$ with heart $\mathcal{H}=\langle S_{0}, \cdots, S_{p-1}\rangle$.
Otherwise, $\mathcal{S}= \mathop{\cup} \limits_{ i=1 }^{n}\{ \tau^{k}S_{a_{i}}\mid0\leq k < l_{i}\}$ as in (\ref{non-empty of simples}). Then each collection $\mathcal{U}$ of simple objects in $\mathcal{Y}$ can be represented by $\mathcal{U}=\mathop{\cup} \limits_{ i=1 }^{n}\mathcal{U}_{i}$ as in (\ref{collection of X}) and
$\mathcal{X}=\mathop{\cup}\limits_{ i=1 }^{n}\mathcal{X}_{i}$  is a SMC in $\mathcal{D}^{b} (${\rm{\textbf{T}}}$_p)$  by Lemma \ref{construct SMCs}.
Note that $j^{*}X=0$ with $X\in\mathcal{U}^{\bot_{\mathcal{D}}}=
\langle\mathcal{S}\rangle^{\bot_{\mathcal{D}}}$ and
$j^{*}F(S^{(t)}_{j})=j^{*}j_{*}S^{(t)}_{j}=S^{(t)}_{j}$ with $S^{(t)}_{j} \in\langle\mathcal{S}\rangle_{\mathcal{D}}$  by Lemma \ref{calculation lemma}.
It follows that $j^{*}\mathcal{X}=\mathcal{U}$ and
$\mathcal{X}=\mathop{\cup}\limits_{ i=1 }^{n}\mathcal{X}_{i}$ is the collection simple objects in heart $\mathcal{H}$.
We are done.
\end{proof}

The following result shows each heart in $\mathcal{D}^{b} (${\rm{\textbf{T}}}$_p)$ precisely has a tube subcategory with maximal rank.

\begin{prop} \label{collection of simples of sub-tube}
Up to shift actions of $\mathcal{D}^{b} (${\rm{\textbf{T}}}$_p)$, each SMC $\mathcal{X}$ in $\mathcal{D}^{b} (${\rm{\textbf{T}}}$_p)$ precisely has  a non-empty maximal subset $\mathcal{X}_{\tube}$ as a  collection of simple objects in a tube subcategory {\rm{\textbf{T}}}$_{\lvert\mathcal{X}_{\tube}\rvert}:=\langle\mathcal{X}_{\rm{tube}}\rangle$ such that
\begin{itemize}
\item [(1)] $\mathcal{X}_{\tube}:=\mathcal{X}$ if $\mathcal{X}=\{S_{0}, S_{1},  \cdots, S_{p-1}\}$;
\item [(2)] $\mathcal{X}_{\tube}:=\mathop{\cup}\limits_{i=1 }^{n}\{S^{(l_{i}+1)}_{a_{i}},  S_{b}\mid a_{i-1}<b<a_{i}-l_{i}\}$ if $\mathcal{X}=\mathop{\cup}\limits_{ i=1 }^{n}\mathcal{X}_{i}$ as in (\ref{collection of X}).
\end{itemize}
\end{prop}

\begin{proof}
Recall that a tube category $\textbf{T}_r$ with the simples $T_{j}$\, $(0\leq j\leq r-1)$ satisfies $\Hom^{1}(T_{j+1}, T_{j})\cong \textbf{k}$.
It follows that  $T_{j}\in\textbf{T}_p[k]$ for all $j$  and some $k\in\mathbb{Z}$ since $\Hom^{1}(S^{(t_{1})}_{j_{1}}[-1], S^{(t_{2})}_{j_{2}})=0$ for any $S^{(t_{1})}_{j_{1}}, S^{(t_{2})}_{j_{2}} \in\textbf{T}_p$.

Clearly, the only candidate of $\mathcal{X}_{\tube}\subseteq\mathcal{X}$ is  $\mathcal{X}_{\tube}=\mathcal{X}$ if $\mathcal{X}=\{S_{0}, S_{1}, \cdots, S_{p-1}\}$.
Otherwise, we have $\mathcal{X}=\mathop{\cup}\limits_{ i=1 }^{n}\mathcal{X}_{i}$ as in (\ref{collection of X}) by Proposition \ref{heart forms}.
We cliam $F(S^{(t_{m})}_{a_{i}})[k_{m}]\not\in\mathcal{X}_{\tube}$ for any $F(S^{(t_{m})}_{a_{i}})[k_{m}]\in\mathcal{X}$.
For contradiction, up to $\tau$-actions of $\textbf{T}_p$, we let $T_{1}=F(S^{(t_{m})}_{a_{i}})[k_{m}]=
S^{(l_{i}+1-t_{m})}_{a_{i}-t_{m}}[k_{m}+1]\in\mathcal{X}_{\tube}$ with $0\leq a_{i}\leq p-1$.
It follows that $\lvert\mathcal{X}_{\tube}\rvert\geq1$ since $\Ext^{1}(T_{1}, T_{1})=0$ with $l_{i}+1-t_{m}<p$.
Hence, there exists $T_{0}=S^{(t_{m'})}_{j}[k_{m}+1]\in\mathcal{X}_{\tube}$ with $\Hom^{1}( S^{(l_{i}+1-t_{m})}_{a_{i}-t_{m}},$ $S^{(t_{m'})}_{j})\neq0$ and $\Hom^{0}(S^{(t_{m'})}_{j}, S^{(l_{i}+1-t_{m})}_{a_{i}-t_{m}})= 0$.
By similar arguments as in Lemma \ref{vanising homs}, we have $T_{0}=S^{(t_{m'})}_{a_{i}-l_{i}-1}[k_{m}+1]$.
But $S^{(t_{m'})}_{a_{i}-l_{i}-1}[k_{m}+1]\not\in\mathcal{X}$ with $k_{m}+1\geq 1$ if  $S_{a_{i-1}}=\tau S_{a_{i}-l_{i}}$, and
$S^{(t_{m'})}_{a_{i}-l_{i}-1}[k_{m}+1]\not\in\mathcal{X}$ if else.
This proves our claim.

Similarly, we get $S^{(t_{m})}_{a_{i}})[k_{m}],\,S^{(t_{m'})}_{j_{m'}})[k_{m'}]
\not\in\mathcal{X}_{\tube}$ for any $S^{(t_{m})}_{a_{i}})[k_{m}],\,S^{(t_{m'})}_{j_{m'}})[k_{m'}]\in\mathcal{X}$
with $k_{m}<0$ and $a_{i}-l_{i}<j_{m'}<a_{i}$.
On the other hand, $\mathcal{X}':=\mathop{\cup}\limits_{i=1 }^{n}\{S^{(l_{i}+1)}_{a_{i}},  S_{b}\mid a_{i-1}<b<a_{i}-l_{i}\}$ is a collection of simple objects in a tube subcategory $\textbf{T}_{p-\lvert \mathcal{S}\rvert}=\mathcal{S}^{\bot_{\textbf{T}_p}}$ by similar arguments as above.
Therefore, the unique candidate  of $\mathcal{X}_{\tube}\subseteq\mathcal{X}$ with maximal rank is  $\mathcal{X}_{\tube}=\mathcal{X}'$ if $\mathcal{X}=\mathop{\cup}\limits_{ i=1 }^{n}\mathcal{X}_{i}$.
We are done.
\end{proof}

By similar arguments as in Lemma \ref{thick condition}, the following result shows a SMC  in $\mathcal{D}^{b} (\textbf{T}_p)$ can be represented by a collection of SMCs of type $\mathbb{A}$.

\begin{cor} \label{decomposition of SMC of rank geq 2}
Let
$\mathcal{S}$ as in (\ref{non-empty of simples}) with $1<\lvert\mathcal{S}\rvert<p-1$,
$\mathcal{A}_{b}=\{S_{b}\}$ with $a_{i-1}<b<a_{i}-l_{i}$  and $\mathcal{A}_{i}=\mathcal{X}_{i}\setminus(\mathop{\cup}\limits_{b=a_{i-1}+1 }^{a_{i}-l_{i}-1}\mathcal{A}_{b})$.
Then $\mathcal{X}=\mathop{\cup}\limits_{i=1 }^{n}\mathcal{X}_{i}=\mathop{\cup}\limits_{i=1 }^{n}(\mathcal{A}_{i}\cup(\mathop{\cup}\limits_{b=a_{i-1}+1 }^{a_{i}-l_{i}-1}\mathcal{A}_{b}))$ with each of $\mathcal{A}_{i}$ and $\mathcal{A}_{b}$ is a SMC of type $\mathbb{A}$.
\end{cor}

Now let us define pre-simple-minded collections of type $\mathbb{A}$ in $\mathcal{D}=\mathcal{D}^{b} (\textbf{T}_p)$ as follows.
\begin{defn} ~\label{defn for pre-smc}
Let $\mathcal{S}$ be a (possibly empty) proper collection of simple objects in ${\emph{\textbf{T}}_p}$.
A collection  $\mathcal{U}$ of objects in $\mathcal{D}^{b} ({\emph{\textbf{T}}_p})$ is said to be  a \emph{pre-simple-minded collection} \emph{(pre-SMC)} of type $\mathbb{A}$ if $\mathcal{U}$ is a \emph{SMC} in $\langle\mathcal{S}\rangle_{\mathcal{D}}$.
\end{defn}

Note that the AR-translation $\tau$ is an auto-equivalence of \rm{\textbf{T}}$_p$.
Hence $\mathcal{U}$ is a pre-SMC in $\mathcal{D}^{b} ({\textbf{T}}_p)$ if and only if $\tau \mathcal{U}$ is a pre-SMC.
Therefore, in the following we only classify the pre-SMCs of type $\mathbb{A}$ up to $\tau$-actions.

\begin{exam}
Up to $\tau$-actions of \rm{\textbf{T}}$_p$, there are five  classes of pre-SMCs of type $\mathbb{A}$ in $\mathcal{D}^{b} ({\textbf{T}_3})$ as in Table 1 (each box determines an equivalence class), where $\delta\in\mathbb{Z},\; k\geq1$.
\begin{table}[h]
\caption{Pre-SMCs\, $\mathcal{U}$ of type $\mathbb{A}$ in $\mathcal{D}^{b} ({\textbf{T}_3})$}
\begin{tabular}{|c|c|c|c|c|c|}
\hline
\multicolumn{5}{|c|}{$\mathcal{U}$}\\
\hline
\makecell*[c]{\; }&&&&$\{S_{2}[\delta],\;  S_{1}[\delta-k+1]\}$\\
\cline{5-5}
\makecell*[c]{$\emptyset$}&&$\{S^{}_{2}[\delta]\}$&
&$\{S^{(2)}_{2}[\delta],\;  S_{1}[\delta+k]\}$\\
\cline{5-5}
\makecell*[c]{\; }& &&
&$\{S^{(2)}_{2}[\delta],\;  S_{2}[\delta-k]\}$  \\
\hline
\end{tabular}
\label{pre-smc in T3}
\end{table}
\end{exam}

Given $\mathcal{S}= \mathop{\cup} \limits_{ i=1 }^{n} \{  \tau^{k}S_{a_{i}}\mid0\leq k < l_{i}  \}$ as in (\ref{non-empty of simples}), we denote $F^{-1}$ by the following assignment
$$F^{-1}\,:\langle\mathcal{S}\cap\tau\mathcal{S}\rangle_{\mathcal{D}}\cup\left(  \langle\tau\mathcal{S}\rangle_{\mathcal{D}}\setminus\langle \mathcal{S}\cap\tau\mathcal{S}\rangle_{\mathcal{D}}\right ) \longrightarrow\langle\mathcal{S}\rangle_{\mathcal{D}}$$
satisfying
\begin{equation}\label{F^{-1} action}
F^{-1}(S^{(t)}_{j}[m])=
\begin{cases}
S^{(l_{i}+1-t_{})}_{a_{i}}[m-1], & \text{if $\soc (S^{(t)}_{j})=S^{}_{a_{i}-l_{i}};$}\\
S^{(t)}_{j}[m], & \text{if else.}\\
\end{cases}
\end{equation}

\begin{lem} ~\label{F^{-1} smcs}
Let $\mathcal{X}\neq\{S_{0}, S_{1}, \cdots, S_{p-1}\}$ be a SMC in $\mathcal{D}^{b}(\emph{\textbf{T}}_p)$ with non-empty maximal subset $\mathcal{X}_{\tube}$ as in Proposition \ref{collection of simples of sub-tube}, and $F^{-1}$ as in (\ref{F^{-1} action}).
Then $F^{-1}(\mathcal{X}\setminus\mathcal{X}_{\tube})$ is a pre-SMC of type $\mathbb{A}$ in $\mathcal{D}^{b}(\emph{\textbf{T}}_p)$.
\end{lem}

\begin{proof}
By Proposition \ref{collection of simples of sub-tube}, we have  $\langle \mathcal{X}_{\tube}\rangle=\mathcal{S}^{\bot_{\textbf{T}_p}}$, where $\mathcal{S}$ is a collection  of  simple objects in $\textbf{T}_p$.
Then $F^{-1}(\mathcal{X}\setminus\mathcal{X}_{\tube})$ is a SMC in $\langle\mathcal{S}\rangle_{\mathcal{D}}$ by Proposition \ref{heart forms}.
We are done.
\end{proof}

Using pre-SMC approach, we provide a classification of hearts  in $\mathcal{D}=\mathcal{D}^{b}({\textbf{T}}_p)$.

\begin{thm} ~\label{bijection of hearts and pre-smcs}
There is a bijection
$$
\xymatrix@C=4.5em@R=5.5ex@M=1pt@!0{
&\{\emph{Hearts}\ \mathcal{H}\ \emph{in}\  \mathcal{D}^{b}({\emph{\textbf{T}}_p})\}\,/\,\mathbb{Z}\quad\quad\quad&
\ar[r]^-{\psi}\quad&&
\quad\quad\quad\quad\quad\quad\quad\quad\{\emph{Pre-SMCs}\ \, \mathcal{U}\ \emph{of}\ \emph{type}\ \mathbb{A}\ \emph{in}\ \mathcal{D}^{b}({\emph{\textbf{T}}_p})\}&&&   \\ }
$$
such that the map $\psi$ sends the hearts $\mathcal{H}$ to pre-SMC\; $\mathcal{U}$ of type $\mathbb{A}$ via the assignment $F^{-1}$ as in (\ref{F^{-1} action}).
\end{thm}

\begin{proof}
By Propositions \ref{heart forms} and \ref{collection of simples of sub-tube}, up to shift actions, each heart $\mathcal{H}$ is generated by a  SMC  $\mathcal{X}=\{S^{(t_{1})}_{j_{1}},\cdots, S^{(t_{r})}_{j_{r}}; S^{(t_{m})}_{j_{m}}[k_{m}]\mid r+1\leq m\leq p\}$ such that  $\mathcal{X}_{\tube}=\{S^{(t_{1})}_{j_{1}},\cdots, S^{(t_{r})}_{j_{r}}\}$ is a collection of simple objects in a tube subcategory $\textbf{T}_r\subseteq\textbf{T}_p$. If $\mathcal{X}=\{S_{0}, S_{1}, \cdots, S_{p-1}\}$, we define $\psi(\mathcal{H})= \emptyset$. Then $\emptyset$ is a SMC\, in $\langle\mathcal{S}\rangle_{\mathcal{D}}$ with  $\mathcal{S}= \emptyset$.
Otherwise, there exist a non-empty collection $\mathcal{S}=\mathop{\cup}\limits_{i=1 }^{n}\{\tau^{k}S_{a_{i}} \mid0\leq k < l_{i}\}$ of  simple objects in $\textbf{T}_p$ such that  $\langle \mathcal{X}_{\tube}\rangle=\mathcal{S}^{\bot_{\textbf{T}_p}}$ by Proposition \ref{collection of simples of sub-tube}.
We define $\psi(\mathcal{H})=F^{-1}(\mathcal{X}\setminus\{S^{(t_{1})}_{j_{1}},\cdots, S^{(t_{r})}_{j_{r}}\})=\{F^{-1}(S^{(t_{m})}_{j_{m}}[k_{m}])\mid r+1\leq m\leq p\}$.
Then $\psi(\mathcal{H})$ is a pre-SMC of type $\mathbb{A}$ in $\mathcal{D}^{b} ({\textbf{T}}_p)$ by Lemma \ref{F^{-1} smcs}.

Conversely, let $\mathcal{U}$ be a pre-SMC of type $\mathbb{A}$. If $\mathcal{U}=\emptyset$, we define $\psi^{-1}(\mathcal{U})=\mathcal{H}=\langle S_{0}, S_{1}, \cdots, S_{p-1}\rangle$. Then $\mathcal{H}$ is a heart with SMC $\mathcal{X}=\{S_{0}, S_{1}, \cdots, S_{p-1}\}$ in $\mathcal{D}^{b} (\textbf{T}_p)$.
Otherwise, there exists a collection $\mathcal{S}=\mathop{\cup}\limits_{i=1 }^{n}\{\tau^{k}S_{a_{i}} \mid0\leq k<l_{i}\}$ of  simple objects in $\textbf{T}_p$ such that ${}^{\bot_{\mathcal{D} }}(<\mathcal{U}>^{\bot_{\mathcal{D}}})=<\mathcal{S}>_{\mathcal{D}}=\thick(\mathcal{U})$ and $\mathcal{U}$ is a SMC in $\langle\mathcal{S}\rangle_{\mathcal{D}}$.
We define $\psi^{-1}(\mathcal{U})=\mathcal{H}=\langle\mathcal{X}\rangle$ with $\mathcal{X}=\mathop{\cup}\limits_{i=1 }^{n}\mathcal{X}_{i}$
as in (\ref{collection of X}).
Then $\mathcal{H}$ is a heart with SMC $\mathcal{X}=\mathop{\cup}\limits_{i=1 }^{n}\mathcal{X}_{i}$ in $\mathcal{D}^{b} ({\textbf{T}}_p)$ by Lemma \ref{construct SMCs}.
Clearly, $\psi\psi^{-1}(\mathcal{U})=\mathcal{U}, \psi^{-1}\psi(\mathcal{H})=\mathcal{H}$.
We are done.
\end{proof}

As an explanation of Theorem \ref{bijection of hearts and pre-smcs}, we list all the hearts in $\mathcal{D}^{b} (${\rm{\textbf{T}}}$_3)$ as below.

\begin{exam} \label{hearts for T3}
Up to $\tau$-actions of $\textbf{T}_3$ and shift actions of $\mathcal{D}^{b} (${\rm{\textbf{T}}}$_3)$, the hearts $\mathcal{H}$  in $\mathcal{D}^{b} ({\textbf{T}_3})$ are classified as in Table 2 (each box determines an equivalence class), where $k\geq1,\; \delta_{1}\geq0,  \delta_{2}<0,\;  \delta_{3}\geq k,\; 0\leq\delta_{4}<k$.
\begin{table}[h]
\caption{Hearts $\mathcal{H}$ in $\mathcal{D}^{b} (${\rm{\textbf{T}}}$_3)$}
\begin{tabular}{|c|c|c|c|c|c|}
\hline
\multicolumn{5}{|c|}{$\mathcal{H}$}\\
\hline
\makecell*[c]{$\langle S^{}_{0},\;  S^{}_{1},\; S^{}_{2}\rangle$}&&$\langle S_{2}^{(3)};\; S^{(2)}_{1}[\delta_{1}+1],\;  S_{1}[\delta_{1}+1-k] \rangle$&&$\langle S_{2}^{(3)};\; S^{(2)}_{2}[\delta_{2}],\;  S_{1}[\delta_{2}+k] \rangle$ \\
\cline{1-1} \cline{3-3} \cline{5-5}
\makecell*[c]{$\langle S_{2}^{(2)},\; S_{0};\; S_{1}[\delta_{1}+1]\rangle$}&&$\langle S_{2}^{(3)};\; S^{}_{2}[\delta_{2}],\;  S_{1}[\delta_{2}+1-k] \rangle$&&$\langle S_{2}^{(3)};\; S^{}_{0}[\delta_{3}+1],\;  S^{(2)}_{1}[\delta_{3}+1-k] \rangle$ \\
\cline{1-1} \cline{3-3} \cline{5-5}
\makecell*[c]{\multirow{2}{4.5cm}{$\langle S_{2}^{(2)},\; S_{0};\; S_{2}[\delta_{2}]\rangle$}}&&\multirow{2}{4.5cm}{$\langle S_{2}^{(3)};\; S^{}_{0}[\delta_{1}+1],\;  S_{1}[\delta_{1}+k] \rangle$}&&$\langle S_{2}^{(3)};\; S^{}_{0}[\delta_{4}+1],\;  S^{}_{2}[\delta_{4}-k] \rangle$ \\
 \cline{5-5}
\makecell*[c]{\;}&&&&$\langle S_{2}^{(3)};\; S_{2}^{(2)}[\delta_{2}],\;  S^{}_{2}[\delta_{2}-k] \rangle$ \\
\hline
\end{tabular}
\label{non-trivial hearts for T3}
\end{table}
\end{exam}

We end this section by providing a concrete example of AR-quiver of a heart in $\mathcal{D}^{b} (${\rm{\textbf{T}}}$_7)$.

\begin{exam}
The  AR-quiver $\Gamma(\mathcal{H}_{})$ of a heart $\mathcal{H}_{}=\langle  S^{(2)}_{1}, S^{(4)}_{5}, S_{6}\, ; S_{0}[1],  S^{}_{3}[-1], S^{}_{4}[-1], S^{}_{5}[-1]  \rangle$ in $\mathcal{D}^{b} (${\rm{\textbf{T}}}$_7)$ has the following form, where the objects in tube subcategoty $\langle S^{(2)}_{1}, S^{(4)}_{5}, S_{6}\rangle$ have been marked by $\bigcirc$.
$$
\xymatrix@C=4.6em@R=5.5ex@M=0.001pt@!0{
&  & S^{}_{3}[-1]\ar[d] &  &  & &  &  &  &  &   &  &  &   \\
&  & S^{(2)}_{4}[-1]\ar[r]\ar[d] & S^{}_{4}[-1]\ar[d] &  &  &  &   &   &   &    &   &   &    \\
&  & S^{(3)}_{5}[-1]\ar[r] & S^{(2)}_{5}[-1]\ar[r] & S^{}_{5}[-1] &   &    &    &    &   &   &    &   &    \\
&\vdots & \vdots & \vdots & \vdots & \vdots & \vdots & \vdots & \vdots & &    &   &   &    \\
& \vdots & \vdots & \vdots & \vdots & \vdots & \vdots & \vdots & \vdots &&   &   &   &    \\
& *++=[o][F]\txt{$S^{(2)}_{1}$}\ar[r]\ar[d] & S^{(3)}_{2}\ar[r]\ar[d] & S^{(4)}_{3}\ar[r]\ar[d] &  S^{(5)}_{4}\ar[r]\ar[d] & *++=[o][F]\txt{$S^{(6)}_{5}$}\ar[r]\ar[d] &
*++=[o][F]\txt{$S^{(7)}_{6}$}\ar[r]\ar[d]  &
*++=[o][F]\txt{$S^{(9)}_{1}$}\ar[d] & \cdots  &   &   &    &       \\
&  S^{}_{1}\ar[r] & S^{(2)}_{2}\ar[r]\ar[d] & S^{(3)}_{3}\ar[r]\ar[d] &  S^{(4)}_{4}\ar[r]\ar[d] & S^{(5)}_{5}\ar[r]\ar[d] &S^{(6)}_{6}\ar[r]\ar[d]  & S^{(8)}_{1}\ar[d] & \cdots &S_{0}[1] &   &    &     \\
& & S_{2}\ar[r] & S^{(2)}_{3}\ar[r] & S^{(3)}_{4}\ar[r] &*++=[o][F]\txt{$S^{(4)}_{5}$}\ar[r] &
*++=[o][F]\txt{$S^{(5)}_{6}$}\ar[r]\ar[d]  &
*++=[o][F]\txt{$S^{(7)}_{1}$}\ar[d] & \cdots  &    &   &   &    \\
&    &  &  &   &   & *++=[o][F]\txt{$S_{6}$}\ar[r]  &
*++=[o][F]\txt{$S^{(3)}_{1}$}\ar[d] & \cdots &    &   &   &   \\
&  &   &  &   &   &   & *++=[o][F]\txt{$S^{(2)}_{1}$}\ar[d] & \cdots  &    &   &   &   \\
& &  &  &  & &  & S_{1} & \cdots  &  &   &  &  &   \\
   }
$$
\end{exam}

\section{Exchange graphs}

In this section we investigate the mutations of SMCs in $\mathcal{D}^{b} (${\rm{\textbf{T}}}$_p)$. We will prove the exchange graph for tube category $\textbf{T}_{p}$ is connected.
Moreover we show the space of  stability conditions on $\mathcal{D}^{b}(${\rm{\textbf{T}}}$_p)$ is connected.

\subsection{SMC mutations}
One useful feature of SMCs is that they admit a notion of mutation.
The precise procedure is as follows, cf. \cite{kysi}.
Let $\mathcal{X}_{i}$ denote the extension closure of $X_{i}$ in $\mathcal{D}=\mathcal{D}^{b} (\textbf{T}_p)$. Assume that for any $j$ the object $X_{j}[-1]$ admits a minimal left approximation (m.l.a for short) $g_{j} : X_{j}[-1] \rightarrow X_{ij}$ in $\mathcal{X}_{i}$.

\emph{The left mutation} $\mu_{i}^{+} (X_{1}, X_{2}, \cdots, X_{r})$ of $X_{1}, X_{2}, \cdots, X_{r}$ at $X_{i}$ is a new collection $X'_{1}, X'_{2}, \cdots,  X'_{r}$ such that $X_{i}' = X_{i}[1]$ and $X_{j}' (j \not= i)$ is the cone of the above left approximation
$$ X_{j}[-1] \stackrel{g_{j}}{\longrightarrow} X_{ij}.$$
Similarly one defines \emph{the right mutation} $\mu_{i}^{-} (X_{1},\ \cdots,\ X_{r})$.

In the following we investigate the mutations of SMCs in $\mathcal{D}^{b} (\textbf{T}_p)$, which are important for our later use.

\begin{lem} \label{mutations of SMC}
Let  $\mathcal{X}=\{X_{1}, X_{2}, \cdots,  X_{p}\}$ be a SMC  in $\mathcal{D}^{b}(${\rm{\textbf{T}}}$_p)$. For any $1\leq i\leq p$, we have
\begin{itemize}
\item [(1)] the clllection $ \mu_{i}^{+} (\mathcal{X})$ is simple
minded;
\item [(2)] the clllection $\mu_{i}^{-}(\mathcal{X})$ is  simple-minded.
\end{itemize}
\end{lem}

\begin{proof}
We only prove the first statement, since the proof for the second one is similar.
Let $X_{i}=S^{(t)}_{j}[m]$. Clearly, $t\leq p$.  We consider the following two cases.
If $X_{i}=S^{(t)}_{j}[m]$ with $t<p$. Then  we have $\Hom(X_{i}, X_{i}[1])=0$ and $\mathcal{X}_{i}=\add(X_{i})$. By \cite[Prop.\;7.6]{kysi}, the conclusion holds true;

Otherwise,  $X_{i}=S^{(t)}_{j}[m]$ with $t=p$. Up to $\tau$-actions of \rm{\textbf{T}}$_p$, we assume  $X_{i}= S_{p-1}^{(p)}$.
Since $\mathcal{X}$ is a SMC, we have $\Hom(X_{j}[-1], S_{p-1}^{(p)})\not=0$ if and only if $X_{j}=S^{(j'+1)}_{j'}[1]$ with  $j\not=i$ and $0\leq j'\leq p-1$.
If $\Hom(X_{j}[-1], S_{p-1}^{(p)})=0$, then the  m.l.a\; $g_{j}$ of  $X_{i}$ is
$X_{j}[-1] \stackrel{g_{j}}{\longrightarrow}0$. Hence $\Hom(g_{j}, S_{p-1}^{(p)})$ and $\Hom(g_{j}, S_{p-1}^{(p)}[1])$ are all injective;
If not, $\Hom(X_{j}[-1], S_{p-1}^{(p)})\not=0$,  we assume  $X_{j}=S^{(j'+1)}_{j'}[1]$ as above. Then the m.l.a\; $g_{j}: S^{(j'+1)}_{j'}\stackrel{}\longrightarrow S_{p-1}^{(p)}$  of  $X_{i}$ is monomorphism.
Since $0\rightarrow S^{(j'+p+1)}_{j'}\rightarrow S_{p-1}^{(2p)}\oplus S^{(j'+1)}_{j'}\rightarrow S_{p-1}^{(p)}\rightarrow0$ is exact, $g_{j}$ yields the follwing pull-back diagram
$$\xymatrix@C=5.2em@R=7ex@M=1pt@!0{
 0\ar[r] &S_{p-1}^{(p)} \ar[r]\ar@{=}[d]& S^{(j'+p+1)}_{j'} \ar[r]\ar[d]& S^{(j'+1)}_{j_{1}}  \ar[d]^-{g_{j}} \ar[r]& 0\\
 0\ar[r] & S_{p-1}^{(p)} \ar[r]&S_{p-1}^{(2p)}\ar[r]& S_{p-1}^{(p)} \ar[r]& 0. }
$$
It follows that the induced maps of $\Hom(g_{j}, S_{p-1}^{(p)} )$ and $\Hom(g_{j}, S_{p-1}^{(p)}[1])$ are both injective.
Therefore,  the conclusion also holds true by \cite[Prop.\;7.6]{kysi}.
We are done.
\end{proof}

\begin{lem}  ~\label{key object}
Let  $\mathcal{A}_{l}=\langle \tau^{l-1} S_{a}, \dots,  \tau S_{a} ,S_{a}\rangle\subsetneq\emph{\textbf{T}}_p$ as in (\ref{non-empty of simples}) with $1\leq l<p$. Then  each SMC\, $\mathcal{U}$  in $\mathcal{D}^{b}(\mathcal{A}_{l})$ has an element $S^{(t)}_{a}[k]$ with $1\leq t\leq l$ and $k\in \mathbb{Z}$.
\end{lem}

\begin{proof}
For contradiction, we let $\mathcal{A}_{l-1}:=\langle \tau^{l-1} S_{a}, \dots,  \tau S_{a}\rangle$. Then we have the thick subcategory  $\thick(\mathcal{U})\subseteq \mathcal{D}^{b} (\mathcal{A}_{l-1})$, a contradiction to $\thick(\mathcal{U}) =\mathcal{D}^{b} ( \mathcal{A}_{l})$. We are done.
\end{proof}

Let $\mathcal{X}= \{ X_{1}, X_{2}, \cdots,  X_{p} \}$ be a SMC in $\mathcal{D}^{b} (${\rm{\textbf{T}}}$_p)$. For convenience, we will simply denote the $l$ times of mutations
$\mu_{k_{l}}^{\pm}\circ\cdots \mu_{k_{2}}^{\pm}\circ\mu_{k_{1}}^{\pm} (\mathcal{X})$ by $\mu_{}^{\pm l} (\mathcal{X})$  if no confusions appear.
Let $\mathcal{X}_{0}=\{S_{0}, S_{1},\cdots, S_{p-1}\}$.

The following result plays a key role in this section.

\begin{prop} ~\label{equations of mutation}
Let $\mathcal{X}=\mathop{\cup}\limits_{i=1}^{n}\mathcal{X}_{i}\neq\mathcal{X}_{0}$
as in (\ref{collection of X}) be a SMC in $\mathcal{D}^{b} (${\rm{\textbf{T}}}$_p)$.
Then we have
\begin{equation} \label{mutated equations}
\mu_{}^{\pm l}(\mathcal{X})=\mathcal{X}_{0},\; l\geq1.
\end{equation}
\end{prop}

\begin{proof}
By Propositions \ref{heart forms} and \ref{collection of simples of sub-tube}, we write $\mathcal{X}=\{S^{(t_{1})}_{j_{1}},\cdots, S^{(t_{r})}_{j_{r}}; S^{(t_{m})}_{j_{m}}[k_{m}]\mid r+1\leq m\leq p\}$ for a SMC such that  $\mathcal{X}_{\tube}=\{S^{(t_{1})}_{j_{1}}, \cdots, S^{(t_{r})}_{j_{r}}\}$ is a collection of simple objects in a tube subcategory $\textbf{T}_r$ with
$$\Hom^{1}(S^{(t_{1})}_{j_{1}}, S^{(t_{r})}_{j_{r}})\cong\Hom^{1}(S^{(t_{m+1})}_{j_{m+1}}, S^{(t_{m})}_{j_{m}})\cong\textbf{k}\, (1\leq m\leq r-1).$$
Since there exists $\mathcal{S}=\mathop{\cup}\limits_{i=1 }^{n}\{\tau^{k}S_{a_{i}} \mid0\leq k < l_{i}\}\neq\emptyset$, we have $S^{(t)}_{a_{i}}[k]\in\mathcal{U}$ as in (\ref{collection of X}) with some $1\leq i\leq n$, $1\leq t\leq l_{i}$ and $k\in\mathbb{Z}$ by Lemma \ref{key object} and hence either $F(S^{(t)}_{a_{i}})[k]\in\mathcal{X}$ with $k\geq0$ or $S^{(t)}_{a_{i}}[k]\in\mathcal{X}$ with $k<0$.
We only prove for the case  $S^{(t)}_{a_{i}}[k]\in\mathcal{X}$ with $k<0$ since the proof for $F(S^{(t)}_{a_{i}})[k]\in\mathcal{X}$ with $k\geq0$ is similar.

We claim there exists $l'\geq1$ such that $\mathcal{X}'_{\tube}\subset\mathcal{X}':=\mu^{+l'}(\mathcal{X})$ with $\lvert\mathcal{X}'_{\tube}\rvert=r+1$.
Since $S^{(l_{i}+1)}_{a_{i}}\in\mathcal{X}_{\tube}$ by Proposition \ref{collection of simples of sub-tube} and the indexes of objects in $\mathcal{X}_{\tube}$ are considered module $r$, we assume   $S^{(t_{i})}_{j_{i}}=S^{(l_{i}+1)}_{a_{i}}$.
Then after  left mutation $-k-1$ times of $\mathcal{X}$  at $S^{(t)}_{a_{i}}[k]$, we have $S^{(l_{i}+1)}_{a_{i}}, S^{(t)}_{a_{i}}[-1] \in\mathcal{X}'':=\mu_{}^{+(-k-1)}(\mathcal{X})$ with $\Hom^{1}(S^{(l_{i}+1)}_{a_{i}}, S^{(t)}_{a_{i}}[-1])\cong\textbf{k}$.

We now consider $\mu_{}^{+} (\mathcal{X}'')$ at $S^{(t)}_{a_{i}}[-1]$.
Then $\mu^{+}(\mathcal{X}'')(S^{(t)}_{a_{i}}[-1])=S^{(t)}_{a_{i}}$.
By the triangle  $\xi_{}: S^{(l_{i}+1)}_{a_{i}}[-1] \stackrel{g_{}}{\longrightarrow} S^{(t)}_{a_{i}}[-1] \rightarrow S^{(l_{i}+1-t)}_{a_{i}-t} \rightarrow S^{(l_{i}+1)}_{a_{i}}$ with the m.l.a\; $g$ of $S^{(l_{i}+1)}_{a_{i}}$ in  $\langle S^{(t)}_{a_{i}}[-1]\rangle$,
we get  $\mu^{+}(\mathcal{X}'')(S^{(l_{i}+1)}_{a_{i}})= S^{(l_{i}+1-t)}_{a_{i}-t}$.
On the other hand, we get $\mu^{+}(\mathcal{X}'')(S^{(t_{m})}_{j_{m}})=S^{(t_{m})}_{j_{m}}$ with $S^{(t_{m})}_{j_{m}}\in\mathcal{X}_{\tube}$ and $m\neq i$ since $\Hom^{0}(S^{(t_{m})}_{j_{m}}[-1], S^{(t)}_{a_{i}}[-1])=0$.
Applying $\Hom(S^{(t)}_{a_{i}}, -)$ to $\xi_{}$, we have an exact sequence
$$
\Hom^{1}(S^{(t)}_{a_{i}}, S^{(l_{i}+1)}_{a_{i}}[-1]) \rightarrow \Hom^{1}(S^{(t)}_{a_{i}}, S^{(t)}_{a_{i}}[-1]) \rightarrow \Hom^{1}(S^{(t)}_{a_{i}}, S^{(l_{i}+1-t)}_{a_{i}-t}) \rightarrow \Hom^{2}(S^{(t)}_{a_{i}}, S^{(l_{i}+1)}_{a_{i}}[-1]).
$$
It follows that
$\Hom^{1}(S^{(t)}_{a_{i}}, S^{(l_{i}+1-t)}_{a_{i}-t})\cong \Hom^{1}(S^{(t)}_{a_{i}}, S^{(t)}_{a_{i}}[-1])\cong\textbf{k}$ since $\Hom^{1}(S^{(t)}_{a_{i}}, S^{(l_{i}+1)}_{a_{i}}[-1])=0=$ $\Hom^{2}(S^{(t)}_{a_{i}}$, $S^{(l_{i}+1)}_{a_{i}}[-1])$.
By similar arguments as above,  we get
$\Hom^{1}(S^{(t_{i+1})}_{j_{i+1}}, S^{(t)}_{a_{i}})\cong \Hom^{1}(S^{(t_{i+1})}_{j_{i+1}}, S^{(t_{i})}_{j_{i}}))\cong\textbf{k}$,
and $\Hom^{1}(S^{(l_{i}+1-t)}_{a_{i}-t}, S^{(t_{i-1})}_{j_{i-1}})\cong\Hom^{0}(S^{(t_{i})}_{j_{i}}[-1],  S^{(t_{i-1})}_{j_{i-1}})\cong\textbf{k}$.
It follows that the subset
$$\{S^{(t_{1})}_{j_{1}}, \cdots, S^{(t_{i-1})}_{j_{i-1}}, S^{(l_{i}+1-t)}_{a_{i}-t}, S^{(t)}_{a_{i}}, S^{(t_{i+1})}_{j_{i+1}}, \cdots,  S^{(t_{r})}_{j_{r}}\}
$$
of $\mathcal{X}':=\mu^{+}(\mathcal{X}'')$ is precisely $\mathcal{X}'_{\tube}$ with $\lvert\mathcal{X}'_{\tube}\rvert=r+1$ as in Proposition \ref{collection of simples of sub-tube}.
This proves our claim.

By replacing $\mathcal{X}$ by $\mathcal{X}'=\mu^{\pm l'}(\mathcal{X})$, and keeping the procedure going on step by
step, we finally obtain  $\mathcal{X}'_{\tube}\subseteq\mathcal{X}$ with $\lvert\mathcal{X}'_{\mathrm{tube}}\rvert=p$.
It follows that there exists $l\geq1$ such that $\mu_{}^{\pm l}(\mathcal{X})=\mathcal{X}_{0}$.
We are done.
\end{proof}

\subsection{2-term SMCs}
Following \cite{brus} a SMC  $\mathcal{X}= \{ X_{1}, X_{2}, \cdots,  X_{p} \}$   in $\mathcal{D}^{b} (\textbf{T}_p)$ is \emph{2-term} if $H^{i}(X_{j})$ vanishes for any $i\neq0,-1$ and any $1\leq j\leq p$, where $H^{i}(X)$ is the $i$-th cohomology object of $X$.
We put
\begin{equation*}
2{-}\rm{SMC}(\textbf{T}_p)= the\ set\ of\ isoclasses\ of\ 2{-}term\ simple{-}minded\ collections\ in\ \mathcal{D}^{b}(\textbf{T}_p).
\end{equation*}

Using the definition of $2$-term SMC, we have the following result.

\begin{lem} ~\label{mutations of 2-SMC}
Let  $\mathcal{X}= \{ X_{1}, X_{2}, \cdots,  X_{p} \}$ be a $2$-term SMC  in $\mathcal{D}^{b} (\emph{\textbf{T}}_p)$. For any $1\leq i\leq p$, we have
\begin{itemize}
\item [(1)] $\mu_{i}^{+} (X)$ is $2$-term if $H^{0}(X_{i})\in\emph{\textbf{T}}_p$;
\item [(2)] $\mu_{i}^{-}(X)$ is $2$-term if $H^{-1}(X_{i})\in\emph{\textbf{T}}_p$.
\end{itemize}
\end{lem}

\subsection{Exchange graphs}

The \emph{exchange graph } $\EG(\textbf{T}_{p})$ of SMCs in $\mathcal{D}^{b}(\textbf{T}_p)$ is
the oriented graph whose vertices are all SMCs in $\mathcal{D}^{b}(\textbf{T}_p)$ and whose edges correspond to left mutations between them, cf.\cite{kingqi}.

We denote $\EG(2$-SMC$(\textbf{T}_p))$ by the full subgraph of $\EG(\textbf{T}_{p})$ with all vertices are all 2-term SMCs in $\mathcal{D}^{b}(\textbf{T}_p)$.
Recall that $\mathcal{X}_{0}=\{S_{0}, S_{1},\cdots, S_{p-1}\}$.

\begin{lem} ~\label{finiteness and connectness of 2-SMCs}
The set $2$-$\emph{SMC}(\emph{\textbf{T}}_p)$ of $2$-term  SMCs  in $\mathcal{D}^{b}(\emph{\textbf{T}}_p)$ is finite.
Moreover, the subgraph \emph{EG}$(2$-$\emph{SMC}(\emph{\textbf{T}}_p))$ of $\EG(\emph{\textbf{T}}_{p})$ is connected.
\end{lem}

\begin{proof}
The first statement follows from the set of indecomposable objects $S^{(t)}_{j}[k]$ with $t\leq p$ and $k=0, 1$ is finite.
Let  $\mathcal{X}=\{S^{(t_{m})}_{j_{m}}[k_{m}]\mid1\leq m\leq p\}$ be a 2-term SMC such that $\mathcal{X}\neq\mathcal{X}_{0}[k]$, where $k=0,1$.
Note that $\mathcal{X}=\mathcal{X}_{0}[k]$ if $k_{m}=k$ for any $1\leq m\leq p$ by Propositions \ref{heart forms} and \ref{collection of simples of sub-tube}.
Then there exists some $1\leq m\leq p$ such that $k_{m}=0$. By Lemma \ref{mutations of 2-SMC}, $\mu_{m}^{+}(\mathcal{X})$ is 2-term.
Since $\mu_{m'}^{+}(\mu_{m}^{+} (\mathcal{X}))(S^{(t_{m})}_{j_{m}})\neq S^{(t_{m})}_{j_{m}}$ for any $m'\neq m$ and then $\mu_{m'}^{+}(\mu_{m}^{+} (\mathcal{X}))\neq\mathcal{X}$, we can keep left mutations of $\mathcal{X}$ at $S^{(t_{m})}_{j_{m}}[k_{m}]$ with $k_{m}=0$ step by step.
It follows that $\mu_{}^{+l}(\mathcal{X})=\mathcal{X}_{0}[1]$ for some $l\geq1$ since $2$-SMC$(\textbf{T}_p)$ is finite.
Similarly, we have $\mu_{}^{-l}(\mathcal{X})=\mathcal{X}_{0}$ for some $l\geq1$.
We are done.
\end{proof}

Exchange graph of hearts is a key point in the study of Bridgeland stability conditions.
Note that $\EG(Q)$ is connected if $Q$ is a quiver  of Dynkin type, while it is disconnected  if $Q$ is non-Dynkin, cf. \cite{qsta}.

We are now in a position to show the connectedness of $\EG(\textbf{T}_p)$.

\begin{thm} \label{connected of EG(Tp)}
The  exchange  graph $\EG(\emph{\textbf{T}}_p)$ of SMCs in $\mathcal{D}^{b} (${\rm{\textbf{T}}}$_p)$ is connected.
\end{thm}

\begin{proof}
Up to shift actions, each SMC $\mathcal{X}\neq\mathcal{X}_{0}$ in $\EG(\textbf{T}_p)/ \mathbb{Z}$ can be representated by $\mathcal{X}=\mathop{\cup}\limits_{i=1}^{n}\mathcal{X}_{i}$
as in (\ref{collection of X}) by Proposition \ref{heart forms}.
Since $\mu_{}^{\pm l}(\mathcal{X})=\mathcal{X}_{0}$ and $\mu_{}^{\pm l} (\mathcal{X}[1])=\mathcal{X}_{0}[1]$ for some $l\geq1$ by Proposition \ref{equations of mutation} and $\mu_{}^{+l'}(\mathcal{X}_{0})=\mathcal{X}_{0}[1]$ for some $l'\geq1$ by Lemma \ref{finiteness and connectness of 2-SMCs}, the subgraph of vertices $\mathcal{X}[m], \mathcal{X}_{0}[m], m\in\mathbb{Z}$ has the form
$$
\xymatrix@C=2.6em@R=5.1ex@M=0.1pt@!0{
\cdots&\mathcal{X}[-1]\ar@{->}[rd]^{\mu_{}^{\pm l}}&&&\mathcal{X}\ar@{->}[d]^-{\mu_{}^{\pm l}}&&
&&\mathcal{X}[1]\ar@{->}[ld]_-{\mu_{}^{\pm l}}&\cdots\\
\cdots&&\mathcal{X}_{0}[-1]\ar@{->}[rr]^-{\mu_{}^{+l'_{}}}&&
\mathcal{X}_{0}\ar@{->}[rrr]^-{\mu_{}^{+l'_{}}}&&&\mathcal{X}_{0}[1]&&\dots
&&  \\ }
$$
It follows that $\EG(\textbf{T}_p)$ is connected and we are done.
\end{proof}

In the following we provide
a concrete example of exchange graph of SMCs in $\mathcal{D}^{b} (${\rm{\textbf{T}}}$_2)$.

\begin{exam}
The  exchange  graph $\EG(\textbf{T}_2)$ of SMCs in $\mathcal{D}^{b} (${\rm{\textbf{T}}}$_2)$ has the form:
$$
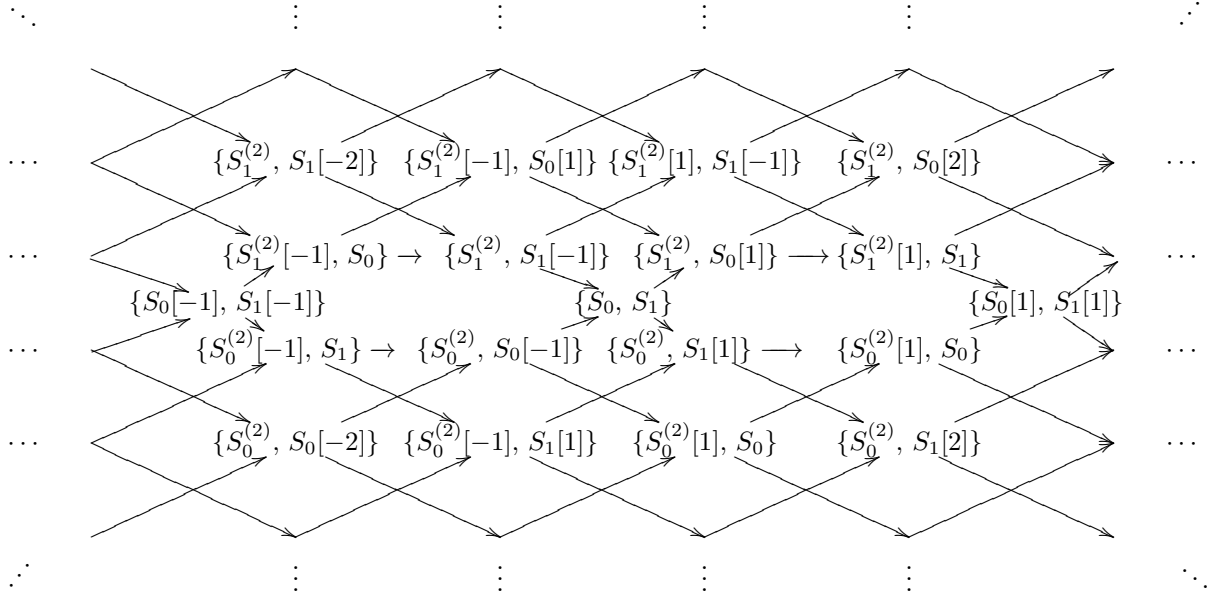
\xymatrix@C=2.55em@R=4.1ex@M=0.001pt@!0{
\ddots&&&&\vdots&&&\vdots&&&\vdots&&&\vdots&&&&
\begin{sideways}$\ddots$\end{sideways}&&&&&  \\
&\ar@{->}[rrrdd]^{}&&&\ar@{->}[rrrdd]^{}&&&\ar@{->}[rrrdd]^{}&&&\ar@{->}[rrrdd]^{}&&&\ar@{->}[rrrdd]^{}&&&&&&& && &  \\
&&&&&&&& && &&&&&&&&& && &  \\
\cdots&\ar@{->}[rrruu]^{}\ar@{->}[rrrdd]^{}&&&\{S^{(2)}_{1},\, S_{1}[-2]\}\ar@{->}[rrruu]^{}\ar@{->}[rrrdd]^{}&&&
\{S^{(2)}_{1}[-1],\, S_{0}[1]\}\ar@{->}[rrruu]^{}\ar@{->}[rrrdd]^{}&&&
\{S^{(2)}_{1}[1],\, S_{1}[-1]\}\ar@{->}[rrruu]^{}\ar@{->}[rrrdd]^{}&&&
\{S^{(2)}_{1},\, S_{0}[2]\}\ar@{->}[rrruu]^{}\ar@{->}[rrrdd]^{}&&&&\cdots&&  \\
&&&&&&&& && &&&&&&&&& && &  \\
\cdots&\ar@{->}[rrruu]^{}\ar@<-0.2ex>[rrd]^{}&&&\quad\quad\{S^{(2)}_{1}[-1],\, S_{0}\}\ar@{->}[rrruu]^{}\xrightarrow{}&&&\quad\quad\{S^{(2)}_{1},\, S_{1}[-1]\}\ar@{->}[rrruu]^{}\ar@<0.1ex>[rrd]^{}&&&
\quad\quad\{S^{(2)}_{1},\, S_{0}[1]\}\ar@{->}[rrruu]^{}\xrightarrow{\quad}&&&
\{S^{(2)}_{1}[1],\, S_{1}\}\ar@{->}[rrruu]^{}\ar@<0.2ex>[rrd]^{}&&&&\cdots&&  \\
&&&\{S_{0}[-1],\, S_{1}[-1]\}\ar@<-0.2ex>[rd]^{}\ar@<0.4ex>[ru]^{}&&&&&&
\{S_{0},\, S_{1}\}\ar@<-0.2ex>[rd]^{}\ar@<0.2ex>[ru]^{}\quad&&&&&&
\{S_{0}[1],\, S_{1}[1]\}\ar@<-0.1ex>[rd]^{}\ar@<-0.6ex>[ru]^{}&&&& \\
\cdots&\ar@{->}[rrrdd]^{}\ar@<-0.2ex>[rru]^{}&&&\{S^{(2)}_{0}[-1],\, S_{1}\}\ar@{->}[rrrdd]^{}\xrightarrow{}&&&\{S^{(2)}_{0},\, S_{0}[-1]\}\ar@{->}[rrrdd]^{}\ar@<-0.02ex>[rru]^{}&&&\{S^{(2)}_{0},\, S_{1}[1]\}\ar@{->}[rrrdd]^{}\xrightarrow{\quad}&&&\{S^{(2)}_{0}[1],\, S_{0}\}\ar@{->}[rrrdd]^{}\ar@<-0.0001ex>[rru]^{}&&&&\cdots&&  \\
&&&&&&&& && &&&&&&&&& && &  \\
\cdots&\ar@{->}[rrrdd]^{}\ar@{->}[rrruu]^{}&&&\{S^{(2)}_{0},\, S_{0}[-2]\}\ar@{->}[rrrdd]^{}\ar@{->}[rrruu]^{}&&&
\ar@{->}[rrrdd]^{}\{S^{(2)}_{0}[-1],\, S_{1}[1]\}\ar@{->}[rrruu]^{} &&&\{S^{(2)}_{0}[1],\,S_{0}\}\ar@{->}[rrrdd]^{}\ar@{->}[rrruu]^{}&&&
\{S^{(2)}_{0},\,S_{1}[2]\}\ar@{->}[rrrdd]^{}\ar@{->}[rrruu]^{}&&&&\cdots&& && &  \\
&&&&&&&& && &&&&&&&&& && &  \\
&\ar@{->}[rrruu]^{}&&&
\ar@{->}[rrruu]^{}&&&\ar@{->}[rrruu]^{}&&&
\ar@{->}[rrruu]^{}&&&\ar@{->}[rrruu]^{}&&&&& && &  \\
&\begin{sideways}$\ddots$\end{sideways}\quad\quad\quad\quad\quad\quad&&&\vdots&&&
\vdots&&&\vdots&&&\vdots&&&&\quad\ddots&& && &  \\ }
$$
It turns out that we can obtain any
SMC in $\mathcal{D}^{b} (${\rm{\textbf{T}}}$_2)$ by a sequence of mutations from $\mathcal{X}_{0}=\{S_{0}, S_{1}\}$.
\end{exam}

\subsection{Bridgeland stability conditions }
As an application, we will show  the space $\Stab(\textbf{T}_p)$ of  Bridgeland stability conditions on $\mathcal{D}^{b}(${\rm{\textbf{T}}}$_p)$ is connected via $\EG(\textbf{T}_p)$.

\begin{lem}\label{finitely many extension closed subcategory in each heart}
There are only finitely many subcategories of each heart in $\mathcal{D}^{b}(\emph{\textbf{T}}_p)$ which are closed under extensions and direct summands.
\end{lem}

\begin{proof}
By similar arguments as in \cite[Lem.  4.5]{stab},
we obtain that $\langle S_{j}^{(rp)}\rangle=\langle S_{j}^{(p)}\rangle$, and $\langle S_{j}^{(rp+l)}\rangle=\langle S_{j}^{(p+l)}\rangle$ for  any $r\geq 1$,  $0\leq j\leq p-1$ and $1\leq l\leq p-1$. Note that there are only finitely many indecomposable objects in $\textbf{T}_p$ with length smaller than $2p$.
Moreover, there are only finitely many subcategories $\textbf{T}_p[k]$ contained in a heart $\mathcal{H}$ since $\mathcal{H}\subsetneq\mathcal{D}^{b}(\textbf{T}_p)=
\mathop{\bigvee}\limits_{k\in\mathbb{Z}}^{}\textbf{T}_p[k]$.
Then the result follows.
\end{proof}

Recall that for any torsion pair $(\mathcal{T}, \mathcal{F})$ in a heart, we have both $\mathcal{T}$ and $\mathcal{F}$ are closed under extensions and direct summands.
Then Lemma \ref{finitely many extension closed subcategory in each heart} implies the following result.

\begin{lem}\label{finitely many torsion pairs in each heart}
There are only finitely many torsion pairs in every heart of $\mathcal{D}^{b}(\emph{\textbf{T}}_p)$.
\end{lem}

By Lemma \ref{bijectionHeartandSmc}, we identify each SMC $\mathcal{X}$ with the associated heart $\mathcal{H}=\langle\mathcal{X}\rangle$ in $\EG(\textbf{T}_p)$.
Let $\EG_{0}$ be a connected component of the exchange graph $\EG(\mathcal{D})$ of a triangulated category $\mathcal{D}$ and $\Stab_{0}(\mathcal{D})=\mathop{\cup}\limits_{\mathcal{H}\in
\EG_{0}}^{}\rm{U}(\mathcal{H})$, where $\rm{U}(\mathcal{H})$ is the set of stability conditions in $\mathcal{D}$ whose heart is $\mathcal{H}$.

We recall the finiteness condition as in \cite[Assump.\;3.1]{qsta} as follows.

\begin{assump}\label{Assump}
Every heart in $\EG_{0}$ is finite and has only finitely many torsion pairs.
\end{assump}

We end this section by proving the connectedness of   $\Stab(\textbf{T}_p)$.

\begin{prop}\label{connected stab space}
$\Stab(\emph{\textbf{T}}_p)$ is connected with the canonical embedding $\EG(\emph{\textbf{T}}_p)\hookrightarrow
\Stab(\emph{\textbf{T}}_p)$.
\end{prop}

\begin{proof}
By Theorem \ref{connected of EG(Tp)}, $\EG(\textbf{T}_p)$ is connected with $\EG(\textbf{T}_p)=\EG_{0}(\textbf{T}_p)$.
Moreover, $\EG(\textbf{T}_p)$ satisfies finiteness condition as in Assumption \ref{Assump} by
Lemmas \ref{recollements for hearts of tubes} and \ref{finitely many torsion pairs in each heart}.
Then $\Stab(\textbf{T}_p)$ is connected with $\Stab(\textbf{T}_p)=
\mathop{\cup}\limits_{\mathcal{H}\in
\EG(\textbf{T}_p)}^{}\rm{U}(\mathcal{H})$  by \cite[Thm.\;2.12]{wool} and
\cite[Thm.\;3.2, Thm.\;3.4]{qsta}.
We are done.
\end{proof}

\section{Ext-quivers}

In this section we  introduce the notion of graded gentle one-cycle quiver.
We will describe all the Ext-quivers of  SMCs in $\mathcal{D}^{b} (${\rm{\textbf{T}}}$_p)$ via the graded gentle one-cycle quivers with $p$ vertices.

\subsection{Ext-quivers}
Let $\mathcal{X}$ be a SMC in  triangulated category $\mathcal{D}$.
Following \cite{kingqi} the \emph{Ext-quiver}  $Q_{\mathcal{X}}$ is the (positively) graded quiver defined as follows:
\begin{itemize}
\item [(1)] The vertices of $Q_{\mathcal{X}}$ are indexed by objects of $\mathcal{X}$;
\item [(2)] For $X_{i}, X_{j} \in \mathcal{X}$, there are $\Dim _{\mathbf{k}} \Hom^{k}(X_{i}, X_{j})$ arrows of degree $k\geq1$ from $X_{i}$ to $X_{j}$.
\end{itemize}

Recall that a graded gentle tree $\mathcal{G}$ for a heart of type $\mathbb{A}$  defined in \cite{qext} is a gentle tree with a positive grading for each arrow. The associated quiver $\mathcal{Q} (\mathcal{G})$ of $\mathcal{G}$, is a graded quiver with the same vertex set and an arrow $\textit{a} :$   $i \rightarrow j$ for each unicolored path $\textit{p} :$ $ i \rightarrow j$ in $\mathcal{G}$, with the natural grading of $\textit{p}$.  \par
For a graded gentle tree $\mathcal{G}$, let $V$ be a vertex with neighborhood

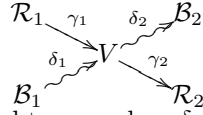
\xymatrix@C=3em@R=3.6ex@M=0.1pt@!0{
&&&&\mathcal{R}_{1}\ar@{->}[rd]^{\gamma_{1}}&&\mathcal{B}_{2}&&&&&&  \\
&&&&&V\ar@{->}[rd]^{\gamma_{2}}\ar@{~>}[ru]^{\delta_{2}}&&&   &&&&  \\
&&&&\mathcal{B}_{1}\ar@{~>}[ru]^{\delta_{1}}& &\mathcal{R}_{2}&&&&  \\ }
where $\mathcal{B}_{i}, \mathcal{R}_{i}$ are the subtrees and $\gamma_{i}, \delta_{i}$ are degrees of $\mathcal{G}$, $i=1,2$. The straight line represents one color and the curly line represent the other color.

Recall that the forward (left) mutation $\mu_{V}$ at  vertex $V$(on $\mathcal{G}$) is defined in \cite{qextA}  as follows:
\begin{itemize}
\item [(1)]  if $\delta_{1} \geq 1$, $\mu_{V}$   on the lower part of the quiver is :
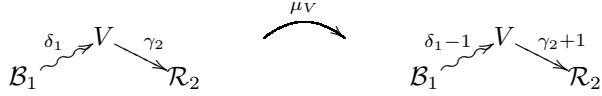
\begin{equation*}
\xymatrix@C=3em@R=3.6ex@M=0.1pt@!0{
& &  V\ar@{->}[rd]^{\gamma_{2}} & &\ar@/^/[r]^{\mu_{V}}& &  &  V\ar@{->}[rd]^{\gamma_{2}+1}&  & &&  & &  \\
& \mathcal{B}_{1} \ar@{~>}[ru]^{\delta_{1}}  & &\mathcal{R}_{2} &  & &\mathcal{B}_{1} \ar@{~>}[ru]^{\delta_{1}-1}  & &\mathcal{R}_{2}  & &  \\  }
\end{equation*}
\item [(2)] if $\delta_{1} = 1$, denote
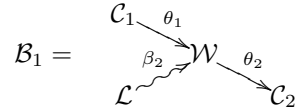
\begin{equation*}
\xymatrix@C=3em@R=3.6ex@M=0.1pt@!0{
&    \mathcal{C}_{1} \ar@{->}[rd]^{\theta_{1}} &  &  &  & &  \\
\mathcal{B}_{1}= & &   \mathcal{W}\ar@{->}[rd]^{\theta_{2}}& &  & &  \\
&  \mathcal{L}\ar@{~>}[ru]^{\beta_{2}} & & \mathcal{C}_{2}   &  &    & &  \\ }
\end{equation*}
and $\mu_{V}$   on the lower part of the quiver is :
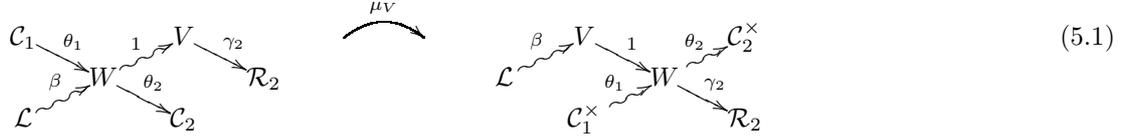
\begin{equation} \label{forward mutation of gentle trees}
\xymatrix@C=3em@R=3.6ex@M=0.1pt@!0{
&\mathcal{C}_{1}\ar@{->}[rd]^{\theta_{1}}&&
V\ar@{->}[rd]^{\gamma_{2}}&&\ar@/^/[r]^{\mu_{V}}&&&
V\ar@{->}[rd]^{1}&&\mathcal{C}^{\times}_{2}&&  \\
&&W\ar@{->}[rd]^{\theta_{2}}\ar@{~>}[ru]^{1}&&\mathcal{R}_{2}  &&&\mathcal{L}\ar@{~>}[ru]^{\beta}&&W \ar@{->}[rd]^{\gamma_{2}}\ar@{~>}[ru]^{\theta_{2}}&&&  \\
&\mathcal{L}\ar@{~>}[ru]^{\beta}&&\mathcal{C}_{2}&&&&& \mathcal{C}^{\times}_{1}\ar@{~>}[ru]^{\theta_{1}}&&\mathcal{R}_{2}&&&  \\ }
\end{equation}
where $\mathfrak{X}^{\times}$ is the operation of swapping colors on a graded gentle trees $\mathfrak{X}$.
\item [(3)]  $\mu_{V}$ on the upper part follows the mirror of the lower part.
\end{itemize}

Dually, define the backward mutation $\mu^{-1}_{V}$ to be the reverse of $\mu_{V}$ (which follows a similar rule).

For a SMC $\mathcal{U}$ of type $\mathbb{A}$ containing a projective-injective object, the Ext-quiver of $\mathcal{U}$ has the following form in the sense of \cite{qext}, which we omit the detailed proof here.

\begin{lem} \label{Ext quiver of type A with pro inj}
Let $\mathcal{A}_{l+1}=\langle \tau^{l} S_{a}, \dots,  \tau S_{a} ,S_{a}\rangle\subsetneq\emph{\textbf{T}}_p$ as in (\ref{non-empty of simples}) and  $\mathcal{U}$ be  SMC\,  in $\mathcal{D}^{b} ( \mathcal{A}_{l+1})$ with $S_{a}^{(l+1)}\in\mathcal{U}$. Then the graded gentle tree $\mathcal{G}_{\mathcal{U}}$ of $\mathcal{U}$ has the form with subtrees $\mathcal{B}_{}, \mathcal{R}_{}$ as below
\begin{equation} \label{graded A form}
\xymatrix@C=4.1em@R=5.2ex@M=0.2pt@!0{
\mathcal{B}_{}\ar@{~>}[r]^{\delta_{}}& S_{a}^{(l+1)}\ar@{->}[r]^{\gamma_{}}&\mathcal{R}_{}.    \\ }
\end{equation}
Moreover, if $S_{a-l+t_{1}}^{(t_{1}+1)}[k_{1}], S_{a}^{(t_{2})}[k_{2}] \in\mathcal{U}$ such that $k_{2}<k<k_{1}$ for any other $S_{a-l+t_{}}^{(t_{}+1)}[k_{}], S_{a}^{(t_{})}[k_{}]\in\mathcal{U}$, then $\mathcal{G}_\mathcal{U}$ has the form with subtrees $\mathcal{B}_{}, \mathcal{R}_{}, \mathcal{B}_{i}, \mathcal{R}_{i}$ as below
\begin{equation} \label{more graded A form}
\xymatrix@C=6.7em@R=5.2ex@M=0.5pt@!0{
&\mathcal{R}_{1}\ar@{->}[d]^-{\gamma_{1}}&&
\mathcal{B}_{1}\ar@{~>}[d]^-{\delta_{1}}\\
\mathcal{B}_{}\ar@{~>}[r]^-{\delta_{}}&S_{a-l+t_{1}}^{(t_{1}+1)}[k_{1}]
\ar@{->}[d]^-{\gamma_{2}}\ar@{~>}[r]^-{k_{1}}&S_{a}^{(l+1)}\ar@{->}[r]^-{-k_{2}}&
S_{a}^{(t_{2})}[k_{2}]\ar@{~>}[d]^-{\delta_{2}}
\ar@{->}[r]^-{\gamma_{}}&\mathcal{R}_{}.\\
&\mathcal{R}_{2}&&\mathcal{B}_{2}\\ }
\end{equation}
\end{lem}

Moreover,  using Lemma \ref{Ext quiver of type A with pro inj} the Ext-quiver of a SMC $\mathcal{U}$ consisting of two parts of type $\mathbb{A}$ is described by the following result.
We omit this proof here.

\begin{lem} \label{Ext quiver of two type A with pro inj}
Let $\mathcal{A}_{l_{i}+1}=\langle \tau^{l_{i}} S_{a_{i}}, \dots,  \tau S_{a_{i}} ,S_{a_{i}}\rangle\subsetneq\emph{\textbf{T}}_p$ as in (\ref{non-empty of simples}) and  $\mathcal{U}_{i}$ be  SMC\,  in $\mathcal{D}^{b}(\mathcal{A}_{l_{i}+1})$ with $S_{a_{1}+1}=\tau^{l_{2}} S_{a_{2}}$, $S_{a_{i}}^{(l_{i}+1)}\in\mathcal{U}_{i}$ and $i=1,\,2$. Then the graded gentle tree $\mathcal{G}_{\mathcal{U}}$ of $\mathcal{U}:=\mathcal{U}_{1}\cup\mathcal{U}_{2}$ in $\mathcal{D}^{b}(\mathcal{A}_{l_{1}+l_{2}+2})$ has one of the form
with subtrees $\mathcal{B}_{i}, \mathcal{R}_{i}$ as below
\begin{equation} \label{more Ext quiver of two type A}
\xymatrix@C=4.5em@R=5.2ex@M=0.5pt@!0{
\mathcal{B}_{1}\ar@{~>}[d]_-{\delta_{1}}&\mathcal{R}_{2}&&
\mathcal{B}_{1}\ar@{~>}[d]_-{\delta_{1}}&\mathcal{R}_{2}  \\
S_{a_{1}}^{(l_{1}+1)}\ar@{->}[d]_-{\gamma_{1}}&
S_{a_{2}}^{(l_{2}+1)}\ar@{~>}[u]_-{\gamma_{2}}\ar@{->}[l]^{1}&&
S_{a_{1}}^{(l_{1}+1)}\ar@{->}[d]_-{\gamma_{1}}\ar@{~>}[r]^{1}&
S_{a_{2}}^{(l_{2}+1)}\ar@{~>}[u]_-{\gamma_{2}}\ar@<1.2ex>[l]^-{1}  \\
\mathcal{R}_{1}&\mathcal{B}_{2},\ar@{->}[u]_-{\delta_{2}}&&
\mathcal{R}_{1}&\mathcal{B}_{2}.\ar@{->}[u]_-{\delta_{2}}  \\
&S_{a_{2}+1}\neq\tau^{l_{1}} S_{a_{1}}\quad\quad\quad&&\quad\quad\quad\quad\quad    S_{a_{2}+1}=\tau^{l_{1}} S_{a_{1}}&&&  \\ }
\end{equation}
\end{lem}

We now introduce graded gentle one-cycle quiver of rank $r$ as follows.

\begin{defn}   \label{generalised graded gentle trees of rank n}
A \emph{graded gentle one-cycle quiver $\widetilde{\mathcal{G}}$ of rank} $rk(\widetilde{\mathcal{G}})=r$ is a graded quiver of the following forms such that
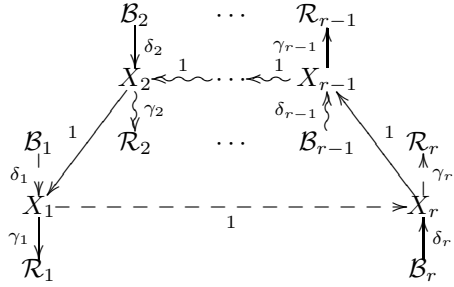
\begin{equation} \label{graph of graded gentle one cycle trees}
\xymatrix@C=3.52em@R=5.6ex@M=0.02pt@!0{
&&&&&\mathcal{B}_{2}\ar@{->}[d]^{\delta_{2}}&\cdots&
\mathcal{R}_{r-1}&&&&&&\\
&&&&&X_{2}\ar@{~>}[d]^{\gamma_{2}}\ar@{->}[ldd]_{1}& \cdots\ar@{~>}[l]_{1}& X_{r-1}\ar@{->}[u]^{\gamma_{r-1}}\ar@{~>}[l]_{1}&&&&&&&&&&  \\
&&&&\mathcal{B}_{1}\ar@{-->}[d]_{\delta_{1}}&\mathcal{R}_{2}&\cdots& \mathcal{B}_{r-1}\ar@{~>}[u]^{\delta_{r-1}}&\mathcal{R}_{r}&&&& \\
&&&&X_{1}\ar@{->}[d]_{\gamma_{1}}\ar@{-->}[rrrr]_-{1}&&&&
X_{r}\ar@{-->}[u]_{\gamma_{r}}\ar@{->}[luu]_{1}&&&&&&&&&& \\
&&&&\mathcal{R}_{1}&&&&
\mathcal{B}_{r}\ar@{->}[u]_{\delta_{r}}&&&&&& \\ }
\end{equation}
\begin{itemize}
\item [(1)] \emph{One graded cycle:} $\widetilde{\mathcal{G}}$ contains exactly one graded cycle $C_{r}$ with all degrees equal one formed by vertices $X_{1}, X_{2}, \cdots$, $X_{r}$  with $r$ alternating colors of  the straight,  curly and dotted line types;

\item [(2)] \emph{Alternating graded gentle trees:} each
\xymatrix@C=2.4em@R=2ex@M=1pt@!0{
\mathcal{G}_{i}:=&\mathcal{B }_{i}\ar@{~>}[r]^-{\delta_{i}}&   X_{i}\ar@{->}[r]^-{\gamma_{i}}&\mathcal{R}_{i}\\ }
is a graded gentle tree with the color of incoming  arrow of $X_{i}$ in $\mathcal{G}_{i}$ is the same to each of outgoing  arrows of $X_{i-1}$ in $\mathcal{G}_{i-1}$ and $X_{i}$ in $C_{r}$, where the index is considered module $r$;
particularly, when $rk(\widetilde{\mathcal{G}})=1$,
\xymatrix@C=2.4em@R=2ex@M=1pt@!0{
\mathcal{G}_{1}:=&\mathcal{B }_{1}\ar@{->}[r]^-{\delta_{1}}&   X_{1}\ar@{->}[r]^-{\gamma_{1}}&\mathcal{R}_{1}\\ }
is a graded gentle tree with colors of the straight line and dotted line, and
\begin{equation*}
\xymatrix@C=3.8em@R=1.8ex@M=1pt@!0{
\widetilde{\mathcal{G}}= \\ }
\xymatrix@C=3.8em@R=1.8ex@M=1pt@!0{
\mathcal{B}_{1}\ar@/_1.6pc/@{{}{x}{}}[rr]
\ar@{->}[r]^-{\delta_{1}}&X_{1}^{\;\ar@(l, u)^-{1}\;\;\;  }\ar@{->}[r]^-{\gamma_{1}}&\mathcal{R}_{1}\\ }
\end{equation*}
such that
\xymatrix@C=3.2em@R=1.3ex@M=0.2pt@!0{
X_{1}^{\;\;\ar@(l, u)^-{2}\;\;\;  }\\ }
is vanishing,
and the composition from $\mathcal{B}_{1}$ to $\mathcal{R}_{1}$ going through \xymatrix@C=3.2em@R=1.3ex@M=0.2pt@!0{
X_{1}^{\;\;\ar@(l, u)^-{1}\;\;\;  }\\ } is valid with degree $\delta_{1}+\gamma_{1}+1$, while the one in
\xymatrix@C=2.4em@R=2.5ex@M=1pt@!0{
\mathcal{B}_{1}\ar@{->}[r]^-{\delta_{1}}&
X_{1}\ar@{->}[r]^-{\gamma_{1}}&\mathcal{R}_{1}\\ } is vanishing in $\widetilde{\mathcal{G}}$.
\end{itemize}
\end{defn}

We now describe the  Ext-quiver of each SMC $\mathcal{X}$ via the graded gentle one-cycle quiver.
Recall that up to  shift actions of $\mathcal{D}^{b} (${\rm{\textbf{T}}}$_p)$, we only  consider $\mathcal{X}=\mathop{\cup}\limits_{i=1}^{n}\mathcal{X}_{i}$ as in (\ref{collection of X}).

\begin{prop}  \label{Ext quiver of rank more than  two}
Let
$\mathcal{S}$ as in (\ref{non-empty of simples}) with $1<\lvert\mathcal{S}\rvert<p-1$ and $\mathcal{X}=\mathop{\cup}\limits_{i=1}^{n}\mathcal{X}_{i}$ as in (\ref{collection of X}).
Then the Ext-quiver of $\mathcal{X}$ in $\mathcal{D}^{b} (${\rm{\textbf{T}}}$_p)$ is  the associated quiver of graded gentle one-cycle quiver $\widetilde{\mathcal{G}}_{}$ with rank $r=p-\lvert \mathcal{S}\rvert$ as in (\ref{graph of graded gentle one cycle trees}).
\end{prop}

\begin{proof}
By Proposition \ref{collection of simples of sub-tube}, we write $\mathcal{X}_{\tube}=\{S^{(t_{1})}_{j_{1}}, \cdots, S^{(t_{r})}_{j_{r}}\}=\mathop{\cup}\limits_{i=1 }^{n}\{S^{(l_{i}+1)}_{a_{i}},  S_{b}\mid a_{i-1}<b<a_{i}-l_{i}\}\subset\mathcal{X}$ for the collection of simple objects in a tube subcategory $\textbf{T}_r\subseteq\textbf{T}_p$ with
$$\Hom^{1}(S^{(t_{1})}_{j_{1}}, S^{(t_{r})}_{j_{r}})\cong\Hom^{1}(S^{(t_{m+1})}_{j_{m+1}}, S^{(t_{m})}_{j_{m}})\cong\textbf{k}\, (1\leq m\leq r-1).$$
By Corollary \ref{decomposition of SMC of rank geq 2},
$\mathcal{X}=\mathop{\cup}\limits_{i=1 }^{n}\mathcal{X}_{i}=\mathop{\cup}\limits_{i=1 }^{n}(\mathcal{A}_{i}\cup(\mathop{\cup}\limits_{b=a_{i-1}+1 }^{a_{i}-l_{i}-1}\mathcal{A}_{b}))$ with each $\mathcal{A}_{i}$ and $\mathcal{A}_{b}$ is a SMC of type $\mathbb{A}$.
We rewrite the index of $\mathcal{A}_{i}$ and $\mathcal{A}_{b}$ such that $\mathcal{X}=\mathop{\cup}\limits_{m=1}^{r}\mathcal{A}_{m}$ with $S^{(t_{m})}_{j_{m}}\in\mathcal{A}_{m}$ when $S^{(t_{m})}_{j_{m}}\in\mathcal{X}_{\tube}$.
Then we get the graded gentle tree of $\mathcal{G}_{\mathcal{U}_{m}}$  with $\mathcal{U}_{m}:=\mathcal{A}_{m}\cup\mathcal{A}_{m+1}$ is as in (\ref{more Ext quiver of two type A}) and the Ext-quiver of $\mathcal{U}_{m}$ is the associated quiver of $\mathcal{G}_{\mathcal{U}_{m}}$ by Lemma \ref{Ext quiver of two type A with pro inj}.
Similarly as in Lemma  \ref{vanising homs}, we have $\Hom^{\bullet}(S^{(t_{})}_{j_{}}[k_{}], S^{(t')}_{j'}[k'])=0=\Hom^{\bullet}(S^{(t')}_{j'}[k'], S^{(t_{})}_{j_{}}[k_{}])$ with $S^{(t_{})}_{j_{}}[k_{}]\in \mathcal{D}^{b}(\mathcal{A}_{m}), S^{(t')}_{j'}[k']\in \mathcal{D}^{b}(\mathcal{A}_{m+2})$.
Then there are no  paths of positive degrees between $\mathcal{A}_{m}$ and $\mathcal{A}_{m+2}$ with $r\geq3$.
By combining the above arguments, we are done.
\end{proof}

Using Lemmas \ref{isom homs} and \ref{Ext quiver of type A with pro inj}, we have the following lemma, cf. \cite[Thm. 2.11]{qext}.

\begin{lem} \label{Ext quiver of F(U)}
Let $\mathcal{A}_{p-1}=\langle S_{1}, S_{2}, \dots, S_{p-1}\rangle\subsetneq\emph{\textbf{T}}_p$, $\mathcal{U}$ be  SMC\,  in $\mathcal{D}^{b}(\mathcal{A}_{p-1})$ with $S_{p-1}^{(t_{1})}[k_{1}], S_{p-1}^{(t_{2})}[k_{2}]\in\mathcal{U}$, $k_{2}<0<k_{1}$ such that $k_{2}<k<k_{1}$ for any other $S_{p-1}^{(t_{})}[k_{}]\in\mathcal{U}$, and $F(\mathcal{U})$ as in (\ref{F(U_i)and X_i}).
Then $F(\mathcal{U})$ is a SMC in $\mathcal{D}^{b}(\mathcal{A}_{p-1})$ and $\mathcal{G}_{F(\mathcal{U})}$ has the form induced from $\mathcal{G}_{\mathcal{U}}$  with subtrees $\mathcal{B}_{}, \mathcal{R}_{}, \mathcal{B}_{i}, \mathcal{R}_{i}$ as below
\begin{equation} \label{more graded F(U) form}
\xymatrix@C=7.95em@R=5.4ex@M=0.2pt@!0{
&\mathcal{R}_{1}\ar@{~>}[d]^-{\gamma_{1}}&
\mathcal{B}_{1}\ar@{~>}[d]^-{\delta_{1}}\\
\mathcal{B}_{}\ar@{->}[r]^-{\delta_{}}&F(S_{p-1}^{(t_{1})})[k_{1}]
\ar@{~>}[d]^-{\gamma_{2}}\ar@{->}[r]^-{k_{1}-k_{2}+2}&
S_{p-1}^{(t_{2})}[k_{2}]\ar@{~>}[d]^-{\delta_{2}}
\ar@{->}[r]^-{\gamma_{}}&\mathcal{R}_{}.\\
&\mathcal{R}_{2}&\mathcal{B}_{2}\\ }
\end{equation}
\end{lem}

\begin{prop} \label{Ext quiver of rank one}
Let
$\mathcal{S}=\{\tau^{p-2}S_{a_{}}, \cdots, \tau^{}S_{a_{}},  S_{a_{}}\}$ and $\mathcal{X}=\mathop{\cup}\limits_{i=1}^{n}\mathcal{X}_{i}$ as in (\ref{collection of X}).
Then the Ext-quiver of $\mathcal{X}$ in $\mathcal{D}^{b} (${\rm{\textbf{T}}}$_p)$ is  the associated quiver of graded gentle one-cycle quiver $\widetilde{\mathcal{G}}_{}$ with rank $r=1$ as in (\ref{graph of graded gentle one cycle trees}).
\end{prop}

\begin{proof}
Up to $\tau$-actions of \rm{\textbf{T}}$_p$, we assume $S^{}_{a_{}}=S_{p-1}^{}$.
Then $S_{p-1}^{(p)}\in\mathcal{X}$  and $\Hom^{d}(S_{p-1}^{(p)}, S_{p-1}^{(p)})\cong\textbf{k}$ with $d>0$ if and only if $d=1$.
By Lemma \ref{Ext quiver of F(U)}, the Ext quiver of  $F(\mathcal{U})$ in $\mathcal{D}^{b}(\mathcal{A}_{p-1})$ is the associated quiver of $\mathcal{G}_{F(\mathcal{U})}$ as in (\ref{more graded F(U) form}).
Let $S_{j}^{(t)}[k]\in\mathcal{X}$ and $t<p-t_{1},  t<t_{2}$, then
\begin{equation*}
\Hom^{\bullet}(S_{p-1}^{(p)}, S_{j}^{(t)}[k])\neq0\iff S_{j}^{(t)}=S_{p-1}^{(t)}\iff\Hom^{1}( F(S_{p-1}^{(t_{1})}), S_{j}^{(t)})\neq0\neq\Hom^{0}(S_{p-1}^{(t_{2})}, S_{j}^{(t)});
\end{equation*}
\begin{equation*}
\xymatrix@C=4.8em@R=3.5ex@M=1pt@!0{
\Hom^{\bullet}(S_{j}^{(t)}[k], S_{p-1}^{(p)})\neq0\iff S_{j}^{(t)}[1]=F(S_{p-1}^{(p-t)})\iff\Hom^{-1}(S_{j}^{(t)}, F(S_{p-1}^{(t_{1})}))\neq0\neq\Hom^{1}(S_{j}^{(t)}, S_{p-1}^{(t_{2})}). \\ }
\end{equation*}
It follows that $\Hom^{-k}(S_{p-1}^{(p)}, S_{j}^{(t)}[k])$ $\cong\textbf{k}\cong\Hom^{1-k}(S_{p-1}^{(p)}, S_{j}^{(t)}[k])$ and $\Hom^{1-k}(F(S_{j}^{(t)})[-k], S_{p-1}^{(p)})$
$\cong\textbf{k}\cong\Hom^{2-k}(F(S_{j}^{(t)})[-k], S_{p-1}^{(p)})$ if $S_{j}^{(t)}=S_{p-1}^{(t)}$;
and
$\Hom^{1-k}(F(S_{p-1}^{(t_{1})}), S_{p-1}^{(t_{2})}[k])\cong\textbf{k}$
if
$\Hom^{-k}(S_{p-1}^{(t_{1})}, S_{p-1}^{(t_{2})}[k])$
$\cong\textbf{k}$ with $k<0$ and $t_{2}<t_{1}<p$.
By combining the above arguments, we are done.
\end{proof}

\subsection{Quiver mutations}

Mutation of quivers is defined without reference to any SMC in $\mathcal{D}^{b} (${\rm{\textbf{T}}}$_p)$. It is an interesting problem to
``\,lift\," mutation of quivers to ``\,mutation\," of SMCs associated with the quivers.
Recall that $\mu_{V}\mathcal{G}$ is the forward (left) mutation at vertex  $V\in\mathcal{G}$ on a graded gentle tree $\mathcal{G}$ as in (\ref{forward mutation of gentle trees}).

In the following we define a mutation $\mu_{}$ on graded gentle one-cycle quiver.
\begin{defn}
For a graded gentle one-cycle quiver $\widetilde{\mathcal{G}}$ as in (\ref{graph of graded gentle one cycle trees}), we write
$\mathcal{G}_{i-1,i+2}$ for a graded subquiver of $\widetilde{\mathcal{G}}$ going from $\mathcal{G}_{i-1}$ to $\mathcal{G}_{i+2}$ with  straight line  and curly line colors  such that
\begin{equation} \label{neigbours of V on graded gentle one cycle}
\xymatrix@C=3.5em@R=4.4ex@M=0.5pt@!0{
&&&&\mathcal{B}_{i}\ar@{->}[rd]^-{\delta_{i}}&&\mathcal{R}_{i}&&& \\
&&\widetilde{\mathcal{G}}=&\mathcal{R}_{i+1}&&
X_{i}\ar@{->}[rd]^-{1}\ar@{-->}[ru]^-{\gamma_{i}}&&
&&  \\
&&&\mathcal{B}_{i+1}\ar@{-->}[r]_-{\delta_{i+1}}
&X_{i+1}\ar@{~>}[lu]_-{\gamma_{i+1}}\ar@{-->}[ru]^-{1}&&
\mathcal{G}_{i-1,i+2}\ar@{~>}[ll]^-{1}&&  \\ }
\end{equation}
We define the \emph{left mutation} $\mu_{V}$ at vertex $V\in\mathcal{G}_{i}$ for any $i$ on  $\widetilde{\mathcal{G}}$  as in (\ref{neigbours of V on graded gentle one cycle}) as follows:
\begin{itemize}
\item [(1)] if $V\in\mathcal{R}_{i}$ is a direct successor of $X_{i}$ with the connecting arrow has degree one such that
\begin{equation*}
\xymatrix@C=3.5em@R=4.4ex@M=0.5pt@!0{
&&&&&\mathcal{L}_{V}\ar@{->}[rd]^-{l_{V}}&&\mathcal{D}_{V}&&&  \\
&&&X_{i}\ar@{-->}[r]^-{1}&\mathcal{R}_{i}=&&
V\ar@{->}[rd]^-{r_{V}}\ar@{-->}[ru]^-{d_{V}}&&&&  \\
&&&&&X_{i}\ar@{-->}[ru]^-{1}&&\mathcal{R}_{V}&&&&  \\ }
\end{equation*}
then $\mu_{V}$ on  $\widetilde{\mathcal{G}}$ is:
\begin{equation*}
\xymatrix@C=3.5em@R=4.4ex@M=0.5pt@!0{
&&\mathcal{B}_{i}\ar@{->}[rd]^-{\delta_{i}}&
\mathcal{L}_{V}\ar@{->}[r]^-{l_{V}}&
V\ar@{->}[r]_-{r_{V}}\ar@{-->}[d]^-{d_{V}}&\mathcal{R}_{V}&&&
&\mathcal{L}'_{V}\ar@{->}[rd]^-{l'_{V}}&&\mathcal{D}'_{V}
&(\mathcal{B}_{i}^{\times})^{c}\ar@{~>}[d]^-{\delta_{i}}&&&&&&&  \\
&\mathcal{R}_{i+1}&&X_{i}\ar@{->}[rd]^-{1}\ar@{-->}[ru]_-{1}&
\mathcal{D}_{V}&&
\ar@/^/[r]^{\mu_{V}}&&
\mathcal{R}^{s}_{i+1}&&
V\ar@{-->}[ru]^-{d'_{V}}\ar@{->}[rr]^-{1}&&
X_{i}\ar@{~>}[ld]^-{1}\ar@{->}[d]^-{r_{V}}&&&  \\
&\mathcal{B}_{i+1}\ar@{-->}[r]_-{\delta_{i+1}}&
X_{i+1}\ar@{~>}[lu]_-{\gamma_{i+1}}\ar@{-->}[ru]^-{1}&&
\mathcal{G}_{i-1,i+2}\ar@{~>}[ll]^-{1}&&&&
\mathcal{B}_{i+1}^{s}\ar@{-->}[r]_-{\delta_{i+1}}&
X_{i+1}\ar@{->}[lu]_-{\gamma_{i+1}}\ar@{-->}[ru]^-{1}&&
\mathcal{G}^{\times}_{i-1,i+2}\ar@{->}[ll]^-{1}&\mathcal{R}_{V}^{c}&&&&&&&&&  \\ }
\end{equation*}
where
\xymatrix@C=2.4em@R=2.8ex@M=1.2pt@!0{
\mathcal{L}'_{V}\ar@{->}[r]^-{l'_{V}}&V\\ }    \xymatrix@C=2.4em@R=2.8ex@M=1.2pt@!0{ \ar@{~>}[r]^-{d'_{V}}&\mathcal{D}'_{V}\\ }
\xymatrix@C=2.4em@R=2.8ex@M=1.2pt@!0{
=\mu_{V}(&\mathcal{L}_{V}\ar@{->}[r]^-{l_{V}}&V\\ }    \xymatrix@C=2.4em@R=2.8ex@M=1.2pt@!0{ \ar@{~>}[r]^-{d_{V}}&\mathcal{D}_{V})\\ }
,
$\mathfrak{X}^{s}$ (resp. $\mathfrak{X}^{c}$) is the operation of replacing the curly (resp. dotted) line color by the straight (resp. curly) one on $\mathfrak{X}$;

\item [(2)] if $V=X_{i}$, then $\mu_{X_{i}}$ on  $\widetilde{\mathcal{G}}$ is:
\begin{equation*}
\xymatrix@C=3.5em@R=4.4ex@M=0.5pt@!0{
&&\mathcal{B}_{i}\ar@{->}[rd]^-{\delta_{i}}&&\mathcal{R}_{i}&&&&&
\mathcal{B}'_{i}\ar@{-->}[rd]^-{\delta'_{i}}&&
\mathcal{R}'_{i}
&&&&&&&&  \\
&\mathcal{R}_{i+1}&&
X_{i}\ar@{-->}[ru]^-{\delta_{i}}\ar@{->}[rd]^-{1}&&
&\ar@/^/[r]^{\mu_{X_{i}}}&&
\mathcal{R}_{i+1}^{s}&&
X_{i}\ar@{->}[ru]^-{\gamma'_{i}}\ar@{-->}[ld]_-{1}&
\mathcal{B}_{i+1}^{s}\ar@{->}[l]^-{\delta_{i+1}}&&&&  \\
&\mathcal{B}_{i+1}\ar@{-->}[r]_-{\delta_{i+1}}&
X_{i+1}\ar@{~>}[lu]_-{\gamma_{i+1}}\ar@{-->}[ru]^-{1}&&
\mathcal{G}_{i-1,i+2}\ar@{~>}[ll]^-{1}&&&&&
X_{i+1}\ar@{->}[lu]_-{\gamma_{i+1}}\ar@{-->}[rr]^-{1}&&
(\mathcal{G}_{i-1,i+2}^{\times})^{d_{i-1}}\ar@<1.2ex>[ll]^-{1}&&&&&&&  \\ }
\end{equation*}
where
\xymatrix@C=2.48em@R=3.1ex@M=1.2pt@!0{
\mathcal{B}'_{i}\ar@{-->}[r]^-{\delta'_{i}}&X_{i}\\ }    \xymatrix@C=2.4em@R=3.1ex@M=1.2pt@!0{ \ar@{->}[r]^-{\gamma'_{i}}&\mathcal{R}'_{i}\\ }
\xymatrix@C=2.4em@R=3.1ex@M=1.2pt@!0{
=(\mu_{X_{i}}(&
\mathcal{B}_{i}\ar@{->}[r]^-{\delta_{i}}&X_{i}\\ }
\xymatrix@C=2.4em@R=3.1ex@M=1.2pt@!0{ \ar@{-->}[r]^-{\gamma_{i}}&\mathcal{R}_{i}))^{\times}\\ }
,
$(\mathcal{G}_{i-1,i+2}^{\times})^{d_{i-1}}$  is the operation of
replacing only the curly line color by the dotted one on the graded subquiver $\mathcal{G}_{i-1}^{\times}$ of\, $\mathcal{G}_{i-1,i+2}^{\times}$;
particularly, when $rk(\widetilde{\mathcal{G}})=1$, we let
$\mathcal{B}'_{1}=\mathcal{B}_{1}$ with $\delta'_{1}=\delta_{1}-1$ and
$\mathcal{R}'_{1}=\mathcal{R}_{1}$ with $\gamma'_{1}=\gamma_{1}+1$ if $\delta_{1}>1$; and
\xymatrix@C=2.4em@R=3.1ex@M=1.2pt@!0{
\mathcal{B}'_{1}\ar@{->}[r]^-{\delta'_{1}}&X_{1}\\ }    \xymatrix@C=2.4em@R=3.1ex@M=1.2pt@!0{ \ar@{-->}[r]^-{1}&\mathcal{B}''_{1}\\ }
\xymatrix@C=2.4em@R=3.1ex@M=1.2pt@!0{
=\mu_{X_{1}}( \\ }
\xymatrix@C=3em@R=3.1ex@M=1.2pt@!0{
\mathcal{B}_{1}\ar@{->}[r]^-{1}&
X_{1})\\ }
and
\xymatrix@C=3em@R=3.1ex@M=1.2pt@!0{
\mathcal{R}'_{1}=&\mathcal{B}''^{\times}_{1}\ar@{->}[r]^-{\gamma_{1}}&
\mathcal{R}_{1}\\ }
with $\gamma'_{1}=1$
if $\delta_{1}=1$, then $\mu_{X_{1}}$ on  $\widetilde{\mathcal{G}}$ is:
\begin{equation*}
\xymatrix@C=3.6em@R=2.7ex@M=0.2pt@!0{
&\mathcal{B}_{1}\ar@/_1.6pc/@{{}{x}{}}[rr]
\ar@{->}[r]^-{\delta_{1}}&X_{1}^{\,\ar@(l, u)^-{1}\;\;\;  }\ar@{->}[r]^-{\gamma_{1}}&\mathcal{R}_{1}&&
\ar@/^/[r]^{\mu_{X_{1}}}&&
\mathcal{B}'_{1}\ar@/_1.6pc/@{{}{x}{}}[rr]
\ar@{->}[r]^-{\delta'_{1}}&X_{1}^{\,\ar@(l, u)^-{1}\;\;\;  }\ar@{->}[r]^-{\gamma'_{1}}&\mathcal{R}'_{1}&&\\
&&&&&&&\\
&&&&&&&\\ }
\end{equation*}
\item [(3)] otherwise,  $\mu_{V}$ on  $\widetilde{\mathcal{G}}$ is:
\begin{equation*}
\xymatrix@C=3.45em@R=4.4ex@M=0.5pt@!0{
&&&\mathcal{B}_{i}\ar@{->}[rd]^-{\delta_{i}}&&\mathcal{R}_{i}&&&&
\mathcal{B}'_{i}\ar@{->}[rd]^-{\delta'_{i}}&&\mathcal{R}'_{i}&&&&&&& \\
&&\mathcal{R}_{i+1}&&
X_{i}\ar@{->}[rd]^-{1}\ar@{-->}[ru]^-{\gamma_{i}}&&
\ar@/^/[r]^{\mu_{V}}&&
\mathcal{R}_{i+1}&&
X_{i}\ar@{->}[rd]^-{1}\ar@{-->}[ru]^-{\gamma'_{i}}&&&&&&&&  \\
&&\mathcal{B}_{i+1}\ar@{-->}[r]_-{\delta_{i+1}}
&X_{i+1}\ar@{~>}[lu]_-{\gamma_{i+1}}\ar@{-->}[ru]^-{1}&&
\mathcal{G}_{i-1,i+2}\ar@{~>}[ll]^-{1}&&&
\mathcal{B}_{i+1}\ar@{-->}[r]_-{\delta_{i+1}}
&X_{i+1}\ar@{~>}[lu]_-{\gamma_{i+1}}\ar@{-->}[ru]^-{1}&&
\mathcal{G}_{i-1,i+2}\ar@{~>}[ll]^-{1}&&&&&&  \\ }
\end{equation*}
where
$\mathcal{R}'_{i}=(\mu_{V}\mathcal{R}_{i})^{\times}$  for $V\in\mathcal{R}_{i}$ (resp. $\mathcal{R}'_{i}=\mathcal{R}_{i}$) with $\gamma'_{i}=\gamma_{i}$ and $\mathcal{B}'_{i}=\mathcal{B}_{i}$ (resp. $\mathcal{B}'_{i}=(\mu_{V}\mathcal{B}_{i})^{\times}$   for $V\in\mathcal{B}_{i}$) with $\delta'_{i}=\delta_{i}$ if  there is a unicolor path between $X_{i}$ and the direct predecessor of $V$ with the connecting arrow of another color has degree one;
and $\mathcal{R}'_{i}=\mu_{V}\mathcal{R}_{i}$ for $V\in\mathcal{R}_{i}$ (resp. $\mathcal{R}'_{i}=\mathcal{R}_{i}$) with $\gamma'_{i}=\gamma_{i}$ and $\mathcal{B}'_{i}=\mathcal{B}_{i}$ (resp. $\mathcal{B}'_{i}=\mu_{V}\mathcal{B}_{i}$  for $V\in\mathcal{B}_{i}$) with $\delta'_{i}=\delta_{i}$ if else.
\end{itemize}
Dually, we define the \emph{right mutation} $\mu^{-1}_{V}$ to be the reverse of $\mu_{V}$  on  $\widetilde{\mathcal{G}}$ (which follows a similar rule).
\end{defn}

Clearly, the set of  all graded gentle one-cycle quivers with $p$ vertices is closed under such mutation.
Moreover by similar arguments as in \cite[Lem.\;2.9]{qext}, we obtain this set is also connected under mutations.

\begin{lem}  \label{connectedness of graded gentle one-cycle quiver}
Any graded gentle one-cycle quiver with $p$ vertices can be repeatedly mutated from another graded gentle one-cycle quiver with $p$ vertices.
\end{lem}

An essential ingredient for quiver mutation is proving reduction to three simple modules.
Then we have to establish a close relationship between arrows and relation spaces, that is, between
$\rm{Ext^{1}}$-groups.

Following \cite{qext}, the next result shows the mutations for objects whose graded quivers are gentle trees.

\begin{lem}  \label{general local mutations}
Let $\mathcal{X}$ be a SMC in $\mathcal{D}^{b}(${\rm{\textbf{T}}}$_p)$ as in (\ref{collection of X}).
For any the following sub-quivers

\xymatrix@C=3em@R=5.4ex@M=0.5pt@!0{
&&&S\ar@{->}[rd]^{a}&&&S\ar@{->}[rd]^{b}&&&S&&&S&&  \\
&&T\ar@{->}[ru]^{1}\ar@{->}[rr]^{a+1}&&A&T\ar@{->}[ru]^{1}&&B&
T\ar@{->}[ru]^{1}&&C\ar@{->}[lu]_-{c+1}\ar@{->}[ll]^{c}
&T\ar@{->}[ru]^{1}&&D\ar@{->}[lu]_-{d+1}&&&  \\ }
\ \
\\
in the Ext-quiver of $\mathcal{Q}(\widetilde{\mathcal{G}}_{\mathcal{X}})$ for some $S, T, A, B, C, D \in \mathcal{X}$ and positive integers $a, b, c, d$,  they become

\xymatrix@C=3em@R=5.4ex@M=0.5pt@!0{
&&&S[1]\ar@{->}[ld]_{1}\ar@{->}[rd]^{a+1}&&&S[1]\ar@{->}[ld]_{1}\ar@{->}[rd]^{b+1}&&&
S[1]\ar@{->}[ld]_{1}&&&S[1]\ar@{->}[ld]_{1}&&  \\
&&R&&A&R\ar@{->}[rr]^{b}&&B&
R&&C\ar@{->}[lu]_-{c}
&R&&D\ar@{->}[lu]_-{d}\ar@{->}[ll]^{d+1}&&&  \\ }
\ \ \\
in the Ext-quiver $\mathcal{Q}(\widetilde{\mathcal{G}}_{\mu^{+}_{S_{}}(\mathcal{X})})$, where $R$ is the non-trivial extension of $T$ on top of $S$.
\end{lem}

The following two lemmas describe the mutations for objects whose graded quivers contain a cycle.

\begin{lem}  \label{local mutations with more cycles}
Let $\mathcal{X}$ be a SMC in $\mathcal{D}^{b} (${\rm{\textbf{T}}}$_p)$ as in (\ref{collection of X}).
For any the following sub-quivers

\xymatrix@C=2.8em@R=6.1ex@M=0.8pt@!0{ &&
&X_{1}\ar@<1.2ex>[r]^-{1}\ar@{->}[d]_-{1}&
X_{2}\ar@{->}[l]^-{1}\ar@{->}[ld]^-{2}&&
X_{1}\ar@{->}[ld]_-{1}\ar@{->}[rrd]_-{a+1}&X_{1}\ar@{->}[l]_-{1}
\ar@{->}[rd]^-{a+2}\ar@{->}[lld]^-{2}&&&
X_{1}\ar@{->}[ld]_-{1}&X_{1}\ar@{->}[l]_-{1}
\ar@{->}[lld]^-{2}&&  \\
&&&S&&S\ar@{->}[rrr]_-{a}&&&A& S&&&B_{}\ar@{->}[lu]_-{b}\ar@{->}[lll]^-{b+2}\ar@{->}[llu]^-{b+1}&&  \\}
\ \ \\
in the Ext-quiver $\mathcal{Q}(\widetilde{\mathcal{G}}_{\mathcal{X}})$ for some $S, X_{1}, X_{2}, A, B \in \mathcal{X}$ and positive integers $a, b$,  they become

\xymatrix@C=2.8em@R=6.1ex@M=0.8pt@!0{ &&
R\ar@{->}[rr]^-{1}&&
X_{2}\ar@{->}[ld]^-{1}&&
R&S[1]\ar@{->}[l]_-{1}\ar@{->}[d]_-{a+1}&R\ar@{->}[l]_-{1}\ar@{->}[ld]^-{a+2}& R&S[1]\ar@{->}[l]_-{1}&R\ar@{->}[l]_-{1}&&&&&  \\
&&&S[1]\ar@{->}[lu]^-{1}&&&&A&&&B\ar@{->}[lu]^{b+2}\ar@{->}[u]_-{b+1}&&&&  \\ }
\ \ \\
in the Ext-quiver $\mathcal{Q}(\widetilde{\mathcal{G}}_{\mu^{+}_{S_{}}(\mathcal{X})})$, where $R$ is the non-trivial extension of $X_{1}$ on top of $S$.
\end{lem}

\begin{proof}
We only prove the statement for the second sub-quiver, since the proofs for the other cases are similar.
Up to $\tau$-actions of \rm{\textbf{T}}$_p$, we assume  $X_{1}=S_{p-1}^{(p)}$.
Note that $\Hom^{1}(S^{(p)}_{p-1},\, S) \cong \textbf{k}$ if and only if $S=S^{(t_{1})}_{p-1}[-1]$; and $\Hom^{a+1}(S^{(p)}_{p-1},\, A) \cong \textbf{k}\cong\Hom^{a+2}(S^{(p)}_{p-1},\, A)$ if and only if $A=S^{(t_{2})}_{p-1}[-a-1]$ with $t_{2}<t_{1}$.
By the triangle  $\xi_{} : \ S^{(p)}_{p-1}[-1] \stackrel{g}{\longrightarrow} S^{(t_{1})}_{p-1}[-1] \rightarrow S^{(p-t_{1})}_{p-1-t_{1}} \rightarrow S^{(p)}_{p-1}$ with the m.l.a\; $g$ of $S^{(p)}_{p-1}$,
we get $\mu^{+}(S^{(t_{1})}_{p-1}[-1])= S^{(t_{1})}_{p-1}$ ,  $\mu^{+}(S^{(p)}_{p-1})=R=: S^{(p-t_{1})}_{p-1-t_{1}}$ and $\mu^{+}(S^{(t_{2})}_{p-1}[-a-1])= S^{(t_{2})}_{p-1}[-a-1]$.
Similarly as in Proposition \ref{equations of mutation},  we get
$\Hom^{1}(S^{(t_{1})}_{p-1}$,\,  $S^{(p-t_{1})}_{p-1-t_{1}})\cong$
$\Hom^{1}(S^{(p-t_{1})}_{p-1-t_{1}},\, S^{(t_{1})}_{p-1})\cong\textbf{k}$, and
$\Hom^{1}(S^{(p-t_{1})}_{p-1-t_{1}},\, S^{(t_{1})}_{p-1})\otimes$
$\Hom^{a+1}(S^{(t_{1})}_{p-1},\, S^{(t_{2})}_{p-1}[-a-1])\cong$
$\Hom^{a+2}(S^{(p-t_{1})}_{p-1-t_{1}},\, S^{(t_{2})}_{p-1}[-a-1])\cong\textbf{k}$.
We are done.
\end{proof}

By similar proofs as in Lemma \ref{local mutations with more cycles}, we have the following result.

\begin{lem}  \label{local mutations with less cycles}
Let $\mathcal{X}$ be a SMC in $\mathcal{D}^{b} (${\rm{\textbf{T}}}$_p)$ as in (\ref{collection of X}).
For any the following sub-quivers

\xymatrix@C=2.8em@R=6.1ex@M=1pt@!0{
&&&X_{1}\ar@<1.2ex>[r]^-{1}\ar@{->}[d]_{a+1}&
X_{2}\ar@{->}[l]^-{1}\ar@{->}[ld]^{a}&
&X_{1}\ar@<1.2ex>[r]^-{1}&
X_{2}\ar@{->}[l]^-{1}&
&X_{1}\ar@{->}[ld]_-{c}&X_{1}\ar@{->}[l]_-{1}\ar@{->}[lld]^-{c+1}&&&
X_{1}&X_{1}\ar@{->}[l]_-{1}&&  \\
&&&A&&&B\ar@{->}[ru]_-{b+2}\ar@{->}[u]^-{b+1}&
&C&&&T\ar@{->}[lu]_-{1}\ar@{->}[llu]^-{2}\ar@{->}[lll]^-{c+2}&
D\ar@{->}[ru]^-{d+2}\ar@{->}[rru]_-{d+1}\ar@{->}[rrr]_-{d}&&&T
\ar@{->}[lu]_-{1}\ar@{->}[llu]^-{2}&& \\ }
\ \ \\
in the Ext-quiver $\mathcal{Q}(\widetilde{\mathcal{G}}_{\mathcal{X}})$ for some $S=X_{1}, X_{2}, T, B, C, D \in\mathcal{X}$ and positive integers $a, b, c, d$,  they become

\xymatrix@C=2.5em@R=6.1ex@M=1pt@!0{
&&
&R\ar@{->}[ld]_-{a}&R\ar@{->}[l]_-{1}\ar@{->}[lld]^-{a+1}&&&
R&R\ar@{->}[l]_-{1}&&
X_{1}[1]\ar@{->}[rrrd]^-{1}\ar@{->}[d]^-{c+1}&&
X_{1}[1]\ar@{->}[rd]^-{2}\ar@{->}[ll]_-{1}\ar@{->}[lld]^-{c+2}&&
X_{1}[1]\ar@{->}[rrrd]^-{1}&&
X_{1}[1]\ar@{->}[ll]_-{1}\ar@{->}[rd]^-{2}&&&&&&\\
&&
A&&&X_{1}[1]\ar@{->}[lu]_-{1}\ar@{->}[llu]^-{2}\ar@{->}[lll]^-{a+2}&
B\ar@{->}[ru]^-{b+2}\ar@{->}[rru]_-{b+1}\ar@{->}[rrr]_-{b}&&&X_{1}[1]
\ar@{->}[lu]_-{1}\ar@{->}[llu]^-{2}&
C&&&R\ar@{->}[lll]^-{c}&
D\ar@{->}[u]_-{d+1}\ar@{->}[rru]_-{d}\ar@{->}[rrr]_-{d+2}&&&R&  \\ }
\ \ \\
in the Ext-quiver $\mathcal{Q}(\widetilde{\mathcal{G}}_{\mu^{+}_{S_{}}(\mathcal{X})})$, where $R$ is the non-trivial extension of $X_{1}$ on top of $X_{2}$ or $T$ respectively.
\end{lem}

Using Lemmas \ref{general local mutations}, \ref{local mutations with more cycles} and \ref{local mutations with less cycles}, a direct consequence of calculation gives the following proposition.

\begin{prop} \label{compatible with quiver mutation}
Let $\widetilde{\mathcal{G}}$ be a graded gentle one-cycle quiver and  $\mathcal{X}$ be a SMC\, in $\mathcal{D}^{b}(${\rm{\textbf{T}}}$_p)$.
If
$\mathcal{Q}(\widetilde{\mathcal{G}})=\mathcal{Q}(\mathcal{X})$  with vertex $V$ in $\widetilde{\mathcal{G}}$ corresponding to the object $S_{}$ in  $\mathcal{X}$, then
\begin{equation}  \label{commutative for quiver mutations}
\mathcal{Q}(\mu^{+}_{S}(\mathcal{X}))=
\mathcal{Q}(\mu^{}_{V}\widetilde{\mathcal{G}}),\quad \mathcal{Q}(\mu^{-}_{S}(\mathcal{X}))=
\mathcal{Q}(\mu^{-1}_{V}\widetilde{\mathcal{G}}).
\end{equation}
\end{prop}

We are now in a position to describe all Ext-quivers of SMCs in $\mathcal{D}^{b} (${\rm{\textbf{T}}}$_p)$.

\begin{thm}
The Ext-quivers of SMCs in $\mathcal{D}^{b}(${\rm{\textbf{T}}}$_p)$ are precisely the associated quivers of graded gentle one-cycle quivers \emph{(}of rank $r$\emph{)} with $p$ vertices  as in (\ref{graph of graded gentle one cycle trees}).
\end{thm}

\begin{proof}
By Proposition \ref{heart forms}, each SMC is a shifts of
$\mathcal{X}_{0}=\{S_{0}, S_{1}, \cdots, S_{p-1}\}$ or $\mathcal{X}=\mathop{\cup}\limits_{ i=1 }^{n}\mathcal{X}_{i}$ as in (\ref{collection of X}).
It follows that the Ext-quiver of any SMC in $\mathcal{D}^{b}(${\rm{\textbf{T}}}$_p)$ is  the associated quiver of graded gentle one-cycle quiver  with $p$ vertices
by Propositions \ref{Ext quiver of rank more than  two} and  \ref{Ext quiver of rank  one}.
Conversely, the graded cyclic quiver of $p$ vertices corresponding to SMC $\mathcal{X}_{0}[k]$ and
the set of  graded gentle one-cycle quivers with $p$ vertices is connected by Lemma \ref{connectedness of graded gentle one-cycle quiver}.
Then we get  the associated quiver of any graded gentle one-cycle quiver with $p$ vertices is the Ext-quiver of some SMC obtained via left/right mutations from $\mathcal{X}_{0}[k]$ by (\ref{commutative for quiver mutations}), (\ref{mutated equations}) and Lemma \ref{mutations of SMC}.
We are done.
\end{proof}

We end this section by providing a concrete example of Ext-quivers of all SMCs  in $\mathcal{D}^{b} ({\textbf{T}_3})$.

\begin{exam}
Recall that, up to $\tau$-actions of $\textbf{T}_3$ and shift actions of $\mathcal{D}^{b} (${\rm{\textbf{T}}}$_3)$, the hearts $\mathcal{H}$  in $\mathcal{D}^{b} ({\textbf{T}_3})$ are classified in Example \ref{hearts for T3}. Then the Ext-quivers of all SMCs are the associated quivers of the following graded gentle one-cycle quivers of rank $r\leq3$, where $k\geq1,\; \delta_{1}\geq0,  \delta_{2}<0,\;  \delta_{3}\geq k,\; 0\leq\delta_{4}<k$.

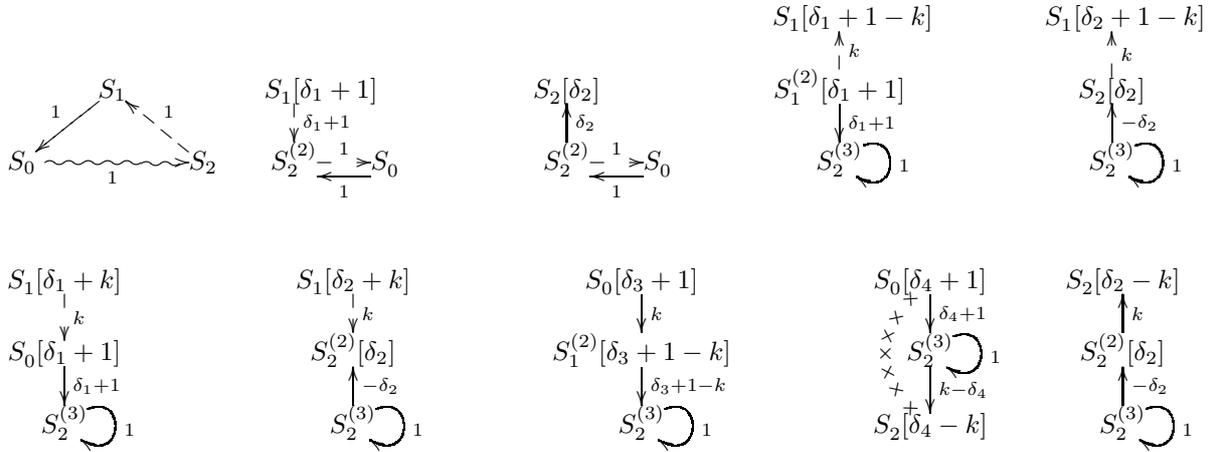
\xymatrix@C=3.4em@R=6.35ex@M=0.1pt@!0{&&&&&&&&&
\quad S^{}_{1}[\delta_{1}+1-k]&&&\quad S^{}_{1}[\delta_{2}+1-k]&&&&&\\
&S_{1}\ar@{->}[ld]_-{1}&&
\quad\quad S^{}_{1}[\delta_{1}+1]\ar@{-->}[d]^-{\delta_{1}+1}&&&
S^{}_{2}[\delta_{2}]&&&
S^{(2)}_{1}[\delta_{1}+1]\ar@{->}[d]^-{\delta_{1}+1}\ar@{-->}[u]_-{k}&&&
S^{}_{2}[\delta_{2}]\ar@{-->}[u]_-{k}&&&&&&&\\
S_{0}\ar@{~>}[rr]_-{1}&&S_{2}\ar@{-->}[lu]_-{1}&
S^{(2)}_{2}\ar@{-->}[r]^-{1}&S^{}_{0}\ar@<1.2ex>[l]^-{1}&&
S^{(2)}_{2}\ar@{-->}[r]^-{1}\ar@{->}[u]_-{\delta_{2}}&
S^{}_{0}\ar@<1.2ex>[l]^-{1}&&
S^{(3)}_{2}\ar@(ur,dr)[]^-{1}&&&
S^{(3)}_{2}\ar@(ur,dr)[]^-{1}\ar@{->}[u]_-{-\delta_{2}}&
&&&&&&&&\\ }

\xymatrix@C=3.6em@R=6.35ex@M=0.01pt@!0{
&&&&&&&&&&&&&&&&&&&&&&&&&\\
S^{}_{1}[\delta_{1}+k]\ar@{-->}[d]^-{k}&&&
S^{}_{1}[\delta_{2}+k]\ar@{-->}[d]^-{k}&&&
S^{}_{0}[\delta_{3}+1]\ar@{->}[d]^-{k}&&&
S^{}_{0}[\delta_{4}+1]\ar@/_1.4pc/ @{{}{x}{}}[dd]\ar@{->}[d]^-{\delta_{4}+1}&&
S^{}_{2}[\delta_{2}-k]&&&&&&&&&&&&&\\
S^{}_{0}[\delta_{1}+1]\ar@{->}[d]^-{\delta_{1}+1}&&&
S^{(2)}_{2}[\delta_{2}]&&&
S^{(2)}_{1}[\delta_{3}+1-k]\ar@{->}[d]^-{\delta_{3}+1-k}&&&
S^{(3)}_{2}\ar@(ur,dr)^-{1}\ar@{->}[d]^-{k-\delta_{4}}&&
S^{(2)}_{2}[\delta_{2}]\ar@{->}[u]_-{k}&&&&&&&&&&&&\\
S^{(3)}_{2}\ar@(ur,dr)^-{1}&&&
S^{(3)}_{2}\ar@(ur,dr)^-{1}\ar@{->}[u]_-{-\delta_{2}}&&&
S^{(3)}_{2}\ar@(ur,dr)[]^-{1}&&&S^{}_{2}[\delta_{4}-k]&&
S^{(3)}_{2}\ar@(ur,rd)[]^-{1}\ar@{->}[u]_-{-\delta_{2}}&&\\ }
\end{exam}

\bigskip

{\bf Acknowledgments.}
The author would like to thank Prof. Yanan Lin and Prof. Shiquan Ruan for  helpful suggestions and discussions.


\begin{thebibliography}{00}

\bibitem{simp}
S.~Al-Nofayee.
\newblock Simple objects in the heart of a t-structure.
\newblock {\em \rm{J. Pure Appl. Algebra 213(2009), no. 1, 54-59.}}

\bibitem{ahge}
I.~Assem and D.~Happel.
\newblock Generalized tilted algebras of type $\textit{A}_{n}$.
\newblock {\em \rm{ Comm. Algebra 9 (1981), no. 20, 2101-2125.}}

\bibitem{chow}
T.~Bachmann, H.~Kong, G.~Wang, and Z.~Xu.
\newblock The chow t-structure on the $\infty$-category of motivic spectra.
\newblock {\em \rm{Ann. of Math. (2)195(2022), no. 2, 707-773.}}

\bibitem{bbd}
A.A. Be\u{\i}linson, J.~Bernstein, and P.~Deligne.
\newblock \rm{Faisceaux pervers}.
\newblock {\em \rm{Analysis and topology on singular spaces, I (Luminy, 1981),
  5-171. Ast\'{e}risque, 100 Soci$\rm{\acute{e}}$t$\rm{\acute{e}}$
  Math$\rm{\acute{e}}$matique de France, Paris, 1982}}.

\bibitem{bris}
T.~Bridgeland.
\newblock Stability conditions on triangulated categories.
\newblock {\em \rm{Ann. of Math. (2)166(2007), no .2, 317-345.}}

\bibitem{brus}
T.~Br$\rm{\ddot{u}}$stle and D.~Yang.
\newblock Ordered exchange graphs.
\newblock {\em \rm{Advances in representation theory of algebras, 135-193. EMS
  Ser. Congr. Rep. European Mathematical Society (EMS), Z$\rm{\ddot{u}}$rich,
  2013.}}

\bibitem{clus}
A.~Buan, R.~Marsh, M.~Reineke, and I.~Reiten.
\newblock Tilting theory and cluster combinatorics.
\newblock {\em \rm{Adv. Math. 204(2006), no. 2, 572-618.}}

\bibitem{geom}
W.~Chang, Y.~Qiu, and X.~Zhang.
\newblock Geometric model for module categories of {D}ynkin quivers via hearts
  of total stability conditions.
\newblock {\em \rm{J. Algebra 638(2024), 57-89.}}

\bibitem{stab}
M.~Chen, Y.~Lin, and S.~Ruan.
\newblock Stability approach to torsion pairs on abelian categories.
\newblock {\em \rm{J. Algebra 636(2023), 560-602.}}

\bibitem{hall}
B.~Deng and J.~Xiao.
\newblock Hall algebras of cyclic quivers and q-deformed {F}ock spaces.
\newblock {\em \rm{J. Algebra 480(2017), 168-208.}}

\bibitem{clusc}
C.~Fu, S.~Geng, and P.~Liu.
\newblock Cluster algebras arising from cluster tubes {I}: integer vectors.
\newblock {\em \rm{Math. Z. 297(2021), no. 3-4, 1793-1824.}}

\bibitem{glri}
W.~Geigle and H.~Lenzing.
\newblock Perpendicular categories with applications to representations and
  sheaves.
\newblock {\em \rm{J. Algebra 144 (1991), no. 2, 273-343.}}

\bibitem{twis}
F.~Genovese, W.~Lowen, and M.~Van~den Bergh.
\newblock t-{S}tructures and twisted complexes on derived injectives.
\newblock {\em \rm{Adv. Math. 387(2021), Paper No. 107826, 70 pp.}}

\bibitem{kim}
J.~Kim.
\newblock Categorical entropy, (co-)t-structures and {ST}-triples.
\newblock {\em \rm{Math. Z.304(2023), no. 3, Paper No. 37, 30 pp.}}

\bibitem{kingqi}
A.~King and Y.~Qiu.
\newblock Cluster exchange groupoids and framed quadratic differentials.
\newblock {\em \rm{Invent. Math. 220(2020), no. 2, 479-523.}}

\bibitem{kingq}
A.~King and Y.~Qiu.
\newblock Exchange graphs and {E}xt quivers.
\newblock {\em \rm{Adv. Math. 285(2015), 1106-1154.}}

\bibitem{kysi}
S.~Koenig and D.~Yang.
\newblock Silting objects, simple-minded collections, $t$-structures and
  co-$t$-structures for finite-dimensional algebras.
\newblock {\em \rm{Doc. Math. 19 (2014), 403-438.}}

\bibitem{gkm}
M.~Lanini and A.~P$\rm{\ddot{u}}$tz.
\newblock {GKM}-theory for torus actions on cyclic quiver {G}rassmannians.
\newblock {\em \rm{Algebra Number Theory 17(2023), no. 12, 2055-2096.}}

\bibitem{lenz}
H.~Lenzing.
\newblock Hereditary categories. in handbook of tilting theory.
\newblock {\em \rm{volume 332 of London Math. Soc. Lecture Note Ser., pages
  105-146. Cambridge Univ. Press, Cambridge, 2007.}}

\bibitem{para}
D.~Orr and M.~Shimozono.
\newblock On cyclic quiver parabolic {K}ostka-{S}hoji polynomials.
\newblock {\em \rm{J. Combin. Theory Ser. A 190(2022), Paper No. 105634, 27
  pp.}}

\bibitem{noet}
C.~Parra and M.~Saor\'{i}n.
\newblock Hearts of t-structures in the derived category of a commutative
  noetherian ring.
\newblock {\em \rm{Trans. Amer. Math. Soc. 369(2017), no. 11, 7789-7827.}}

\bibitem{quan}
F.~Qin.
\newblock Quantum groups via cyclic quiver varieties {I}.
\newblock {\em \rm{Compos. Math. 152(2016), no. 2, 299-326.}}

\bibitem{qsta}
Y.~Qiu.
\newblock Stability conditions and quantum dilogarithm identities for
  $\rm{D}$ynkin quivers.
\newblock {\em \rm{Adv. Math. 269 (2015), 220-264.}}

\bibitem{qext}
Y.~Qiu.
\newblock Ext-quivers of hearts of $\emph{A}$-type and the orientations of
  associahedrons.
\newblock {\em \rm{J.Algebra}}, 393:60--70, 2013.

\bibitem{ring}
C.~M. Ringel.
\newblock Tame algebras and integral quadratic forms.
\newblock {\em \rm{Lecture Notes in Math., 1099 Springer-Verlag, Berlin, 1984,
  xiii+376 pp.}}

\bibitem{rw}
S.~Ruan and X.~Wang.
\newblock t-{S}tabilities for a weighted projective line.
\newblock {\em \rm{Math. Z. 297 (2021), no. 3-4, 1119-1160.}}

\bibitem{sast}
M.~Saor\'{i}n and J.~$\rm{\check{S}}$t$\acute{}$ov\'{i}$\rm{\check{c}}$ek.
\newblock t-{S}tructures with {G}rothendieck hearts via functor categories.
\newblock {\em \rm{Selecta Math. (N.S.)29(2023), no. 5, Paper No. 77.}}

\bibitem{grot}
M.~Saor\'{i}n, J.~$\rm{\check{S}}$t$\acute{}$ov\'{i}$\rm{\check{c}}$ek, and
  S.~Virili.
\newblock t-{S}tructures on stable derivators and {G}rothendieck hearts.
\newblock {\em \rm{Adv. Math.429(2023), Paper No. 109139, 70 pp.}}

\bibitem{sims}
D.~Simson and A.~Skowro$\rm{\acute{n}}$ski.
\newblock Elements of the representation theory of associative algebras. vol.
  2.
\newblock {\em \rm{volume~71 of London Mathematical Society Student Texts.
  Cambridge University Press, Cambridge, 2007.}}

\bibitem{tstr}
D.~Stanley and A.~van Roosmalen.
\newblock t-{S}tructures on hereditary categories.
\newblock {\em \rm{Math. Z. 293(2019), no. 1-2, 731-766.}}

\bibitem{boun}
C.~Sun.
\newblock Bounded t-structures on the bounded derived category of coherent
  sheaves over a weighted projective line.
\newblock {\em \rm{Algebr. Represent. Theory 23 (2020), no. 6, 2167-2235}}.

\bibitem{wool}
J.~Woolf.
\newblock Stability conditions, torsion theories and tilting.
\newblock {\em \rm{J. Lond. Math. Soc. (2) 82 (2010), no. 3, 663-682.}}

\end{thebibliography}
\end{document}